\documentclass[12pt]{article}
\usepackage{amsmath,amsfonts,amsthm,amssymb,mathtools,enumerate}
\usepackage{mathrsfs,graphicx,color,latexsym, tikz, calc}

\usepackage{cite}
\usepackage{xcolor}
\usepackage[colorlinks=true,
linkcolor=blue,citecolor=blue,
urlcolor=green]{hyperref}

\usepackage{caption}
\usetikzlibrary{shadows}
\usetikzlibrary{patterns,arrows,decorations.pathreplacing}
\newtheorem{thm}{Theorem}
\newtheorem{lem}{Lemma}[section]
\newtheorem{prob}{Problem}
\newtheorem{conj}{Conjecture}

\newtheorem{clm}{Claim}[section]

\title{\bf \Large Disproof of a conjecture on the minimum spectral radius and the domination number
\footnote{This work is supported
by National Natural Science Foundation of China (Nos. 12061074; 12171154; 11971274). }}

\author{
{\small Yarong Hu$^{a,c}$,\ \ Zhenzhen Lou$^{b,}$\footnote{Corresponding author.
\newline{\it \hspace*{5mm}Email addresses:} xjdxlzz@163.com (Z. Lou).}, \ \ Qiongxiang Huang$^{a}$}\\[2mm]
\footnotesize $^a$
College of Mathematics and System Science,
Xinjiang University, Urumqi 830046, China\\
\footnotesize $^b$ College of Science, University of Shanghai for Science and Technology,
  Shanghai 200093, China\\
\footnotesize $^c$
School of Mathematics and Information Technology, Yuncheng University, Yuncheng 044000, China\\}

\date{ }

\begin{document}

\maketitle

\begin{abstract}
Let $G_{n,\gamma}$ be
the set of all connected graphs on $n$ vertices with  domination number $\gamma$.
A graph is called a minimizer graph if it attains the minimum spectral radius among
$G_{n,\gamma}$.
Very recently, Liu, Li and Xie [Linear Algebra and its Applications 673 (2023) 233--258] proved that
the minimizer graph over all graphs in $\mathbb{G}_{n,\gamma}$ must be a tree.
Moreover, they  determined  the minimizer graph among $G_{n,\lfloor\frac{n}{2}\rfloor}$ for even $n$,
and posed the conjecture  on the minimizer graph among $G_{n,\lfloor\frac{n}{2}\rfloor}$ for odd $n$.
In this paper, we disprove the conjecture
and completely determine  the unique minimizer graph among $G_{n,\lfloor\frac{n}{2}\rfloor}$ for odd $n$.
\\[1mm]

\end{abstract}

\baselineskip=0.202in

\section{Introduction}

Throughout this paper, we only consider simple and  undirected  graphs. Let
$G$ be a graph with vertex set $V (G)$ and edge set $E(G)$.
A vertex subset $D \subseteq V (G)$  is called a \emph{dominating set} of  $G$ if every vertex in $V (G)\setminus D$ is adjacent
to (dominated by) a vertex in $D$.
 A vertex in the dominating set is called a \emph{dominating vertex}.
The \emph{domination number} of $G$, denoted by $\gamma(G)$, is   the minimum cardinality
of all dominating sets of $G$.
 For a graph $G$,
a dominating set is called a \emph{minimum dominating set} if its cardinality is $\gamma(G)$.
%A well known
%result \cite{Ore-1962} about $\gamma(G)$ is that for a graph $G$ on $n$  vertices containing no isolated vertex, $\gamma(G) \le \frac{n}{2}$.
Let $A(G)$ be the adjacency matrix of $G$.
%is $A(G) = (a_{ij} )_{n\times n}$, where $a_{ij} = 1$ if and only if $v_iv_j \in E(G)$ and
%$a_{ij} = 0$ otherwise.
 The characteristic polynomial $f(G, x) = \det (xI_n -A(G))$ of the adjacency
matrix $A(G)$ of $G$ is called the \emph{characteristic polynomial} of $G$.
%Let $\lambda_1 \ge  \lambda_2 \ge  \cdots \ge \lambda_n$ be eigenvalues of $A(G)$.
The \emph{spectral radius} $\rho(G)$ of $G$ is the largest eigenvalue
%$\lambda_1$
 of $A(G)$.

For a vertex $v \in  V (G)$, the neighborhood of $v$ is the set $N_G(v) = \{u \in V (G) | uv \in E(G)\}$.
The degree of $v \in V (G)$, denoted by $d_G(v)$, is defined as the number of neighbors of $v$ in $G$, i.e.
$d_G(v) = |N_G(v)|$.
Let $\Delta(G)$  be the maximum degree of $G$.
For  $v \in  V (G)$, if $d_G(v)\ge3$ then vertex
 $v$ is called a \emph{branching vertex};
 if $d_G(v)=1$ then vertex
 $v$ is called a \emph{pendant vertex}.
 A vertex $u \in  V (G)$  is
called a \emph{support vertex} if $u$ is the neighbor of a pendant vertex.
 For  $S \subseteq V (G)$, let $G[S]$ be the subgraph of $G$ induced by $S$,
 and let
 $G-S$ or $G-G[S]$ be the subgraph of $G$ by deleting all vertices in $S$ together with their incident edges.
 The \emph{corona $G \circ K_1$}
is the graph obtained from a copy of the graph $G$ by attaching a pendant vertex $v'$ to each vertex $v \in V (G )$.
As usual,
let  $K_n$ and $P_n$ be  respectively the complete graph and the path on $n$ vertices.

Brualdi and Solheid \cite{Brualdi-Solheid} proposed the following general problem, which becomes one of the
classic problems of spectral graph theory:
\begin{prob}\label{prob}
Given a set $\mathbb{G}$ of graphs on $n$ vertices, find $\min\{\rho(G): G \in \mathbb{G}\}$ and $\max\{\rho(G): G \in \mathbb{G}\}$, and characterize
the graphs which achieve the minimum or maximum value.
\end{prob}

The maximization part of this problem has been solved for a number of graph/tree classes with some invariant,
such as  the domination number \cite{Stevanovic-Aouchiche-Hansen-2008,Chen-He-2012},
 the number of cut vertices \cite{Berman-Zhang-2001},
 the clique number \cite{Nikiforov},
the chromatic number \cite{Feng-Li-Zhang-2007},
the matching number \cite{Feng-Yu-Zhang-2007,Guo-Wang-2001},
the diameter \cite{3-153,Guo-Shao-2006}, and
the independence number \cite{3-Stevanovic,Ji-Lu-2016}.
On the opposite side,
the minimization part of this problem
 has been studied, but not so much,
for example, the
independence number \cite{Lou-Guo-2022,Xu-Hong-Shu-Zhai-2009} and
the diameter \cite{Dam-Kooij-2007}.
For giving domination number $\gamma$, the research on the minimization part of Problem \ref{prob} is rare.
 It is trivial for the case $\gamma=\lceil\frac{n}{3}\rceil$, since it
is exactly the domination number of the $n$-vertices path which has the minimum spectral
radius among all connected graphs on $n$ vertices.
For $\gamma=\lfloor\frac{n}{2}\rfloor$, Aouchiche, Hansen and Stevanovi\'c \cite{Aouchiche-Hansen-Stevanovic-2009} determined that $P_{\frac{n}{2}}\circ K_1$ has the minimum spectral
radius among all  trees on even  $n$ vertices.
Very recently, Liu, Li and Xie \cite{Liu-li} showed
that
$P_{\frac{n}{2}}\circ K_1$  also has  the minimum spectral
radius among all  connected graphs on even  $n$ vertices.
In \cite{Liu-li}, the authors also characteristiced
the graph with  the minimum spectral
radius among all  connected graphs on   $n$ vertices with the domination number $3$.
For  the case $\gamma=2$,
 Guan and Ye \cite{Guan-Ye-2011} obtained the unique graph with the
minimum spectral radius in the class of connected graphs on  $n$ vertices.
Let $G_{n,\gamma}$ be
the set of all connected graphs on $n$ vertices with  domination number $\gamma$. Throughout
this paper, we say a graph is a \emph{minimizer graph} if it attains the minimum spectral radius among all graphs in
$G_{n,\gamma}$.
%In this paper, we investigate  the minimizer graph in  $G_{n,\lfloor\frac{n}{2}\rfloor}$ for odd $n$.

Let $P_{d+1}=u_1u_2\cdots u_{d+1}$ be a path with $d+1$ vertices.
For two integers $i,j$ with $0\le j\le i\le d+1$,
let $T_{i,j}^{d+1}$  be the graph on $i+j+d+1$ vertices  obtained from the path $P_{d+1}$ by attaching a pendant vertex to the  first $i$ vertices and the last $j$
vertices of $P_{d+1}$, respectively (see Fig.\ref{T}).
Let $n\ge 3$ be an odd integer.
For any $i\in [\lceil\frac{n-3}{4}\rceil,\frac{n-3}{2}]$,
it is clear that $T_{i,\frac{n-3}{2}-i}^{\frac{n+3}{2}} \in \mathbb{G}_{n,\lfloor\frac{n}{2}\rfloor}$ (see Fig.\ref{T}).

\begin{figure}[h]
  \centering
    \footnotesize
% This is a LaTeX picture output by TeXCAD.
% File name: [Clipboard].
% Version of TeXCAD: 4.3
% Reference / build: 30-Jun-2012 (rev. 105)
% For new versions, check: http://texcad.sf.net/
% Options on the following lines.
%\grade{\on}
%\emlines{\off}
%\epic{\off}
%\beziermacro{\on}
%\reduce{\on}
%\snapping{\on}
%\pvinsert{% Your \input, \def, etc. here}
%\quality{8.000}
%\graddiff{0.005}
%\snapasp{1}
%\zoom{8.0000}
\unitlength 0.9mm % = 2.845pt
\linethickness{0.4pt}
\ifx\plotpoint\undefined\newsavebox{\plotpoint}\fi % GNUPLOT compatibility
\begin{picture}(165,18)(0,0)
\put(92,15){\circle*{2}}
\put(92,15){\line(0,-1){8}}
\put(92,7){\circle*{2}}
\put(92,15){\line(1,0){8}}
\put(100,15){\circle*{2}}
\put(100,15){\line(0,-1){8}}
\put(100,7){\circle*{2}}
\put(116,15){\circle*{2}}
\put(116,15){\line(0,-1){8}}
\put(116,7){\circle*{2}}
%\dashline{1}(100,15)(116,15)
\put(99.93,14.93){\line(1,0){.9412}}
\put(101.812,14.93){\line(1,0){.9412}}
\put(103.694,14.93){\line(1,0){.9412}}
\put(105.577,14.93){\line(1,0){.9412}}
\put(107.459,14.93){\line(1,0){.9412}}
\put(109.341,14.93){\line(1,0){.9412}}
\put(111.224,14.93){\line(1,0){.9412}}
\put(113.106,14.93){\line(1,0){.9412}}
\put(114.989,14.93){\line(1,0){.9412}}
%\end
\put(91,18){$u_{1}$}
\put(99,18){$u_{2}$}
\put(115,18){$u_{i}$}
\put(122,18){$u_{i\!+\!1}$}
\put(130,18){$u_{i\!+\!2}$}
\put(138,18){$u_{i\!+\!3}$}
\put(147,18){$u_{i\!+\!4}$}
\put(162,18){$u_{\frac{n\!+\!3}{2}}$}
\put(131,0){$T^{\frac{n+3}{2}}_{i,\frac{n-3}{2}-i}$}
\put(116,15){\line(1,0){8}}
\put(124,15){\circle*{2}}
\put(124,15){\line(1,0){8}}
\put(132,15){\circle*{2}}
\put(132,15){\line(1,0){8}}
\put(140,15){\circle*{2}}
\put(140,15){\line(1,0){8}}
\put(148,15){\circle*{2}}
\put(148,15){\line(0,-1){8}}
\put(148,7){\circle*{2}}
\put(164,15){\circle*{2}}
\put(164,15){\line(0,-1){8}}
\put(164,7){\circle*{2}}
%\dashline{1}(148,15)(164,15)
\put(147.93,14.93){\line(1,0){.9412}}
\put(149.812,14.93){\line(1,0){.9412}}
\put(151.694,14.93){\line(1,0){.9412}}
\put(153.577,14.93){\line(1,0){.9412}}
\put(155.459,14.93){\line(1,0){.9412}}
\put(157.341,14.93){\line(1,0){.9412}}
\put(159.224,14.93){\line(1,0){.9412}}
\put(161.106,14.93){\line(1,0){.9412}}
\put(162.989,14.93){\line(1,0){.9412}}
%\end
\put(1,15){\circle*{2}}
\put(1,15){\line(0,-1){8}}
\put(1,7){\circle*{2}}
\put(1,15){\line(1,0){8}}
\put(9,15){\circle*{2}}
\put(9,15){\line(0,-1){8}}
\put(9,7){\circle*{2}}
\put(25,15){\circle*{2}}
\put(25,15){\line(0,-1){8}}
\put(25,7){\circle*{2}}
%\dashline{1}(9,15)(25,15)
\put(8.93,14.93){\line(1,0){.9412}}
\put(10.812,14.93){\line(1,0){.9412}}
\put(12.694,14.93){\line(1,0){.9412}}
\put(14.577,14.93){\line(1,0){.9412}}
\put(16.459,14.93){\line(1,0){.9412}}
\put(18.341,14.93){\line(1,0){.9412}}
\put(20.224,14.93){\line(1,0){.9412}}
\put(22.106,14.93){\line(1,0){.9412}}
\put(23.989,14.93){\line(1,0){.9412}}
%\end
\put(0,18){$u_{1}$}
\put(8,18){$u_{2}$}
\put(24,18){$u_{i}$}
\put(31,18){$u_{i\!+\!1}$}
%\put(47,18){$v_{j\!-\!1}$}
\put(55,18){$u_{d\!+\!2\!-\!j}$}
\put(71,18){$u_{d+1}$}
\put(40,0){$T^{d+1}_{i,j}$}
\put(25,15){\line(1,0){8}}
\put(33,15){\circle*{2}}
\put(49,15){\circle*{2}}
\put(49,15){\line(1,0){8}}
\put(57,15){\circle*{2}}
\put(57,15){\line(0,-1){8}}
\put(57,7){\circle*{2}}
\put(73,15){\circle*{2}}
\put(73,15){\line(0,-1){8}}
\put(73,7){\circle*{2}}
%\dashline{1}(57,15)(73,15)
\put(56.93,14.93){\line(1,0){.9412}}
\put(58.812,14.93){\line(1,0){.9412}}
\put(60.694,14.93){\line(1,0){.9412}}
\put(62.577,14.93){\line(1,0){.9412}}
\put(64.459,14.93){\line(1,0){.9412}}
\put(66.341,14.93){\line(1,0){.9412}}
\put(68.224,14.93){\line(1,0){.9412}}
\put(70.106,14.93){\line(1,0){.9412}}
\put(71.989,14.93){\line(1,0){.9412}}
%\end
%\dashline{1}(32,15)(49,15)
\put(31.93,14.93){\line(1,0){.9444}}
\put(33.819,14.93){\line(1,0){.9444}}
\put(35.707,14.93){\line(1,0){.9444}}
\put(37.596,14.93){\line(1,0){.9444}}
\put(39.485,14.93){\line(1,0){.9444}}
\put(41.374,14.93){\line(1,0){.9444}}
\put(43.263,14.93){\line(1,0){.9444}}
\put(45.152,14.93){\line(1,0){.9444}}
\put(47.041,14.93){\line(1,0){.9444}}
%\end
\end{picture}
  \caption{ \footnotesize Graphs $T^{d+1}_{i,j}$ and  $T^{\frac{n+3}{2}}_{i,\frac{n-3}{2}-i}$.}\label{T}
\end{figure}
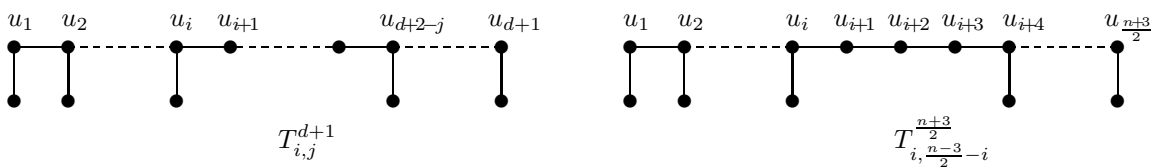

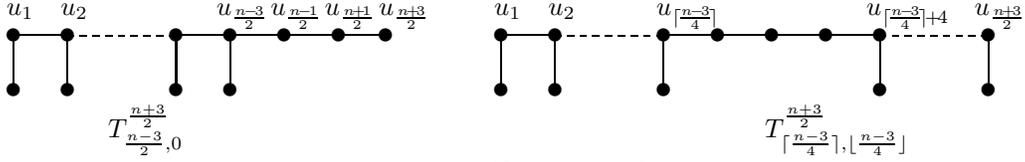
\begin{figure}[h]
  \centering
  \footnotesize
% This is a LaTeX picture output by TeXCAD.
% File name: [Clipboard].
% Version of TeXCAD: 4.3
% Reference / build: 30-Jun-2012 (rev. 105)
% For new versions, check: http://texcad.sf.net/
% Options on the following lines.
%\grade{\on}
%\emlines{\off}
%\epic{\off}
%\beziermacro{\on}
%\reduce{\on}
%\snapping{\on}
%\pvinsert{% Your \input, \def, etc. here}
%\quality{8.000}
%\graddiff{0.005}
%\snapasp{1}
%\zoom{8.0000}
\unitlength 0.9mm % = 2.845pt
\linethickness{0.4pt}
\ifx\plotpoint\undefined\newsavebox{\plotpoint}\fi % GNUPLOT compatibility
\begin{picture}(146,18)(0,0)
\put(1,15){\circle*{2}}
\put(1,15){\line(0,-1){8}}
\put(1,7){\circle*{2}}
\put(1,15){\line(1,0){8}}
\put(9,15){\circle*{2}}
\put(9,15){\line(0,-1){8}}
\put(9,7){\circle*{2}}
\put(25,15){\circle*{2}}
\put(25,15){\line(0,-1){8}}
\put(25,7){\circle*{2}}
%\dashline{1}(9,15)(25,15)
\put(8.93,14.93){\line(1,0){.9412}}
\put(10.812,14.93){\line(1,0){.9412}}
\put(12.694,14.93){\line(1,0){.9412}}
\put(14.577,14.93){\line(1,0){.9412}}
\put(16.459,14.93){\line(1,0){.9412}}
\put(18.341,14.93){\line(1,0){.9412}}
\put(20.224,14.93){\line(1,0){.9412}}
\put(22.106,14.93){\line(1,0){.9412}}
\put(23.989,14.93){\line(1,0){.9412}}
%\end
\put(0,18){$u_{1}$}
\put(8,18){$u_{2}$}
\put(31,18){$u_{\frac{n\!-\!3}{2}}$}
\put(39,18){$u_{\frac{n\!-\!1}{2}}$}
\put(47,18){$u_{\frac{n\!+\!1}{2}}$}
\put(55,18){$u_{\frac{n\!+\!3}{2}}$}
\put(15,0){$T^{\frac{n+3}{2}}_{\frac{n-3}{2},0}$}
\put(25,15){\line(1,0){8}}
\put(33,15){\circle*{2}}
\put(33,15){\line(1,0){8}}
\put(41,15){\circle*{2}}
\put(41,15){\line(1,0){8}}
\put(49,15){\circle*{2}}
\put(33,15){\line(0,-1){8}}
\put(33,7){\circle*{2}}
\put(48,15){\line(1,0){8}}
\put(56,15){\circle*{2}}
\put(73,15){\circle*{2}}
\put(73,15){\line(0,-1){8}}
\put(73,7){\circle*{2}}
\put(73,15){\line(1,0){8}}
\put(81,15){\circle*{2}}
\put(81,15){\line(0,-1){8}}
\put(81,7){\circle*{2}}
\put(97,15){\circle*{2}}
\put(97,15){\line(0,-1){8}}
\put(97,7){\circle*{2}}
%\dashline{1}(81,15)(97,15)
\put(80.93,14.93){\line(1,0){.9412}}
\put(82.812,14.93){\line(1,0){.9412}}
\put(84.694,14.93){\line(1,0){.9412}}
\put(86.577,14.93){\line(1,0){.9412}}
\put(88.459,14.93){\line(1,0){.9412}}
\put(90.341,14.93){\line(1,0){.9412}}
\put(92.224,14.93){\line(1,0){.9412}}
\put(94.106,14.93){\line(1,0){.9412}}
\put(95.989,14.93){\line(1,0){.9412}}
%\end
\put(72,18){$u_{1}$}
\put(80,18){$u_{2}$}
\put(96,18){$u_{\lceil\!\frac{n\!-\!3}{4}\!\rceil}$}
%\put(103,18){$v_{i\!+\!1}$}
%\put(111,18){$v_{i\!+\!2}$}
%\put(119,18){$v_{i\!+\!3}$}
\put(127,18){$u_{\lceil\!\frac{n\!-\!3}{4}\!\rceil\!+\!4}$}
\put(143,18){$u_{\frac{n\!+\!3}{2}}$}
\put(112,0){$T^{\frac{n+3}{2}}_{\lceil\frac{n-3}{4}\rceil, \lfloor\frac{n-3}{4}\rfloor}$}
\put(97,15){\line(1,0){8}}
\put(105,15){\circle*{2}}
\put(105,15){\line(1,0){8}}
\put(113,15){\circle*{2}}
\put(113,15){\line(1,0){8}}
\put(121,15){\circle*{2}}
\put(121,15){\line(1,0){8}}
\put(129,15){\circle*{2}}
\put(129,15){\line(0,-1){8}}
\put(129,7){\circle*{2}}
\put(145,15){\circle*{2}}
\put(145,15){\line(0,-1){8}}
\put(145,7){\circle*{2}}
%\dashline{1}(129,15)(145,15)
\put(128.93,14.93){\line(1,0){.9412}}
\put(130.812,14.93){\line(1,0){.9412}}
\put(132.694,14.93){\line(1,0){.9412}}
\put(134.577,14.93){\line(1,0){.9412}}
\put(136.459,14.93){\line(1,0){.9412}}
\put(138.341,14.93){\line(1,0){.9412}}
\put(140.224,14.93){\line(1,0){.9412}}
\put(142.106,14.93){\line(1,0){.9412}}
\put(143.989,14.93){\line(1,0){.9412}}
%\end
\end{picture}

  \caption{\footnotesize Graphs $T^{\frac{n+3}{2}}_{\frac{n-3}{2},0}$ and
  $T^{\frac{n+3}{2}}_{\lceil\frac{n-3}{4}\rceil, \lfloor\frac{n-3}{4}\rfloor}$.}\label{exT}
\end{figure}

In \cite{Liu-li}, the authors determined the unique minimizer graph in $G_{n,\gamma}$, where $\gamma=\lfloor\frac{n}{2}\rfloor$ and $n$  is even.
\begin{thm}\label{thm-1}(\cite{Liu-li})
Let $n \ge 2$ be an even integer. For any graph $G \in \mathbb{G}_{n,\lfloor\frac{n}{2}\rfloor}$, we have $\rho(G)\ge \rho(T^{\frac{n}{2}}_{\frac{n}{2},0})$ and equality
holds if and only if  $G\cong T^{\frac{n}{2}}_{\frac{n}{2},0}$.
\end{thm}

Meanwhile, the authors \cite{Liu-li} posed a conjecture as follows.

\begin{conj}\label{conj}(\cite{Liu-li})
 Let $n \ge 3$ be odd. For any graph $G \in \mathbb{G}_{n,\lfloor\frac{n}{2}\rfloor}$, then $\rho(G)\ge \rho(T^{\frac{n+3}{2}}_{\frac{n-3}{2},0})$ and equality
holds if and only if  $G\cong T^{\frac{n+3}{2}}_{\frac{n-3}{2},0}$.
\end{conj}

In this paper, we
determine  the minimizer graph among $G_{n,\lfloor\frac{n}{2}\rfloor}$ for odd $n$ (Theorem \ref{extremal-graph}) and then disprove Conjecture \ref{conj}.

\begin{thm}\label{extremal-graph}
 Let $n \ge 3$ be an odd integer. For any graph $G\in \mathbb{G}_{n,\lfloor\frac{n}{2}\rfloor}$, we have $\rho(G)\ge \rho(T^{\frac{n+3}{2}}_{\lceil\frac{n-3}{4}\rceil, \lfloor\frac{n-3}{4}\rfloor })$, with equality
 if and only if  $G\cong T^{\frac{n+3}{2}}_{\lceil\frac{n-3}{4}\rceil, \lfloor\frac{n-3}{4}\rfloor }$.
\end{thm}
\section{Preliminary}
\begin{lem}[\cite{Cvetkovic-Rowlinson-Simic-2010}]\label{subgraph}
 If $H$ is a subgraph of the connected graph $G$, then $\rho(H) \le \rho(G)$. In particular, if $H$ is proper, then
$\rho(H) < \rho(G)$.
\end{lem}

\begin{lem}[\cite{Li-Feng-1979}]\label{lahuan}
Let $v$ be a vertex in a graph $G$ and
suppose that two new paths $P$: $vv_1v_2 \cdots v_k$ and
$Q$: $vu_1u_2\cdots u_m$ of length $k$, $m$
($k \ge m \ge 1$) are attached to $G$ at $v$, respectively, to form a new graph $G_{k,m}$.
Then $\rho(G_{k,m}) >\rho(G_{k+1,m-1})$.
\end{lem}

An internal path of $G$ is a sequence of vertices
 $v_1, v_2, \ldots , v_k$ with $k \ge 2$ such that:\\
(1) the vertices in the sequence are distinct (except possibly $v_1 = v_k$);\\
(2) $v_i$ is adjacent to $v_{i+1}$
($i=1, 2, \ldots , k-1$);\\
(3) the vertex degrees satisfy
$d_G(v_1)\ge 3, d_G(v_2) =\cdots= d_G(v_{k-1}) = 2 $
(unless $k = 2$) and $d_G(v_k) \ge3$.

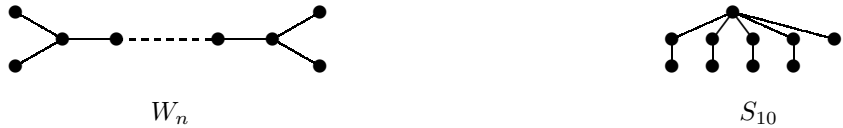
\begin{figure}[h]
  \centering
  \footnotesize
% This is a LaTeX picture output by TeXCAD.
% File name: [Clipboard].
% Version of TeXCAD: 4.3
% Reference / build: 30-Jun-2012 (rev. 105)
% For new versions, check: http://texcad.sf.net/
% Options on the following lines.
%\grade{\on}
%\emlines{\off}
%\epic{\off}
%\beziermacro{\on}
%\reduce{\on}
%\snapping{\on}
%\pvinsert{% Your \input, \def, etc. here}
%\quality{8.000}
%\graddiff{0.005}
%\snapasp{1}
%\zoom{9.5137}
\unitlength 0.9mm % = 2.845pt
\linethickness{0.4pt}
\ifx\plotpoint\undefined\newsavebox{\plotpoint}\fi % GNUPLOT compatibility
\begin{picture}(123,17)(0,0)
%\emline(1,16)(8,12)
\multiput(1,16)(.058823529,-.033613445){119}{\line(1,0){.058823529}}
%\end
\put(1,16){\circle*{2}}
%\emline(8,12)(1,8)
\multiput(8,12)(-.058823529,-.033613445){119}{\line(-1,0){.058823529}}
%\end
\put(1,8){\circle*{2}}
\put(8,12){\circle*{2}}
\put(8,12){\line(1,0){8}}
\put(16,12){\circle*{2}}
\put(31,12){\circle*{2}}
\put(31,12){\line(1,0){8}}
\put(39,12){\circle*{2}}
%\emline(39,12)(46,16)
\multiput(39,12)(.058823529,.033613445){119}{\line(1,0){.058823529}}
%\end
\put(46,16){\circle*{2}}
%\emline(39,12)(46,8)
\multiput(39,12)(.058823529,-.033613445){119}{\line(1,0){.058823529}}
%\end
\put(46,8){\circle*{2}}
%\dashline{1}(16,12)(31,12)
\put(15.93,11.93){\line(1,0){.9375}}
\put(17.805,11.93){\line(1,0){.9375}}
\put(19.68,11.93){\line(1,0){.9375}}
\put(21.555,11.93){\line(1,0){.9375}}
\put(23.43,11.93){\line(1,0){.9375}}
\put(25.305,11.93){\line(1,0){.9375}}
\put(27.18,11.93){\line(1,0){.9375}}
\put(29.055,11.93){\line(1,0){.9375}}
%\end
\put(21,0){$W_n$}
\put(108,0){$S_{10}$}
\put(104,12){\circle*{2}}
\put(104,12){\line(0,-1){4}}
\put(104,8){\circle*{2}}
\put(110,12){\circle*{2}}
\put(110,12){\line(0,-1){4}}
\put(110,8){\circle*{2}}
\put(98,12){\circle*{2}}
\put(98,12){\line(0,-1){4}}
\put(98,8){\circle*{2}}
\put(116,12){\circle*{2}}
\put(116,12){\line(0,-1){4}}
\put(116,8){\circle*{2}}
\put(122,12){\circle*{2}}
\put(107,16){\circle*{2}}
%\emline(107,16)(98,12)
\multiput(107,16)(-.075630252,-.033613445){119}{\line(-1,0){.075630252}}
%\end
\put(107,16){\line(-3,-4){3}}
\put(107,16){\line(3,-4){3}}
%\emline(107,16)(116,12)
\multiput(107,16)(.075630252,-.033613445){119}{\line(1,0){.075630252}}
%\end
%\emline(107,16)(122,12)
\multiput(107,16)(.12605042,-.033613445){119}{\line(1,0){.12605042}}
%\end
\end{picture}

  \caption{\footnotesize Graphs $W_n$ and $S_{10}$.}\label{W-S}
\end{figure}

\begin{lem}[\cite{Hoffman-Smith-1975}]\label{subdivision}
Suppose that $G \not\cong W_n$ (see Fig.\ref{W-S}) and $uv$ is an edge on an internal path of $G$. Let $G_{uv}$ be the graph
obtained from $G$ by the subdivision of the edge $uv$ (i.e., by deleting the edge uv, adding
a new vertex $w$ and
two new edges $uw$ and $wv$). Then $\rho(G_{uv}) <\rho(G)$.
\end{lem}

\begin{lem}[\cite{Cvetkovic-Rowlinson-Simic-2010}]\label{circ}
The spectral radius of the corona $G\circ K_1$ is $\frac{1}{2}(\rho(G)+\sqrt{\rho^2(G)+4})$, where $\rho(G)$ is the spectral radius of  $G$.
\end{lem}

 \begin{lem}[\cite{Ore-1962}]\label{domination-number-upbound}
 For a graph $G$ on $n$  vertices containing no isolated vertex, $\gamma(G) \le \frac{n}{2}$.
\end{lem}

\begin{lem}[\cite{Liu-li}]\label{dominating-set-support}
For a tree $T$, there exists a minimum dominating set $D$ of $T$ such that all support
vertices of $T$ are in $D$.
\end{lem}

 The well-known K\"{o}nig's min-max theorem \cite{Konig1931} states that
the matching number of $G$ is equal to its covering number
 for any bipartite graph $G$.
 It is natural that
the covering number of a graph is less than  its domination number.
Therefore, if $T$ is a tree on even $n$ vertices with the domination number $\gamma(T)=\frac{n}{2}$,
note that the matching number of $T$ is at most $\frac{n}{2}$,
then
$T$ has the perfect matching. Thus, we obtain the following result immediately.
\begin{lem}\label{half-domination}
Let $T$ be a tree on even $n$ vertices with the domination number $\gamma(T)=\frac{n}{2}$.
If $D$ is a minimum dominating set of $T$, then
$V(T)\setminus D$ is also a minimum dominating set of $T$.
\end{lem}

\begin{lem}[\cite{Liu-li}]\label{min-tree}
Let $G$ be the minimizer graph over all graphs in $ \mathbb{G}_{n,\gamma}$. Then $G$ must be a tree.
\end{lem}

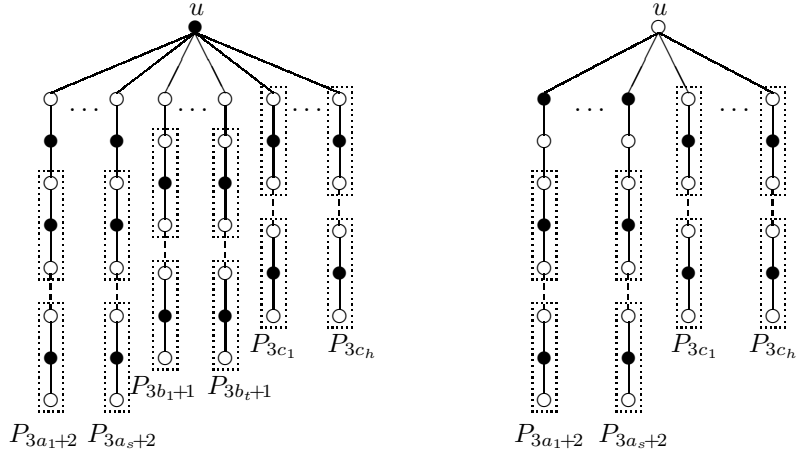
\begin{figure}
  \centering
  \footnotesize
 % This is a LaTeX picture output by TeXCAD.
% File name: [Clipboard].
% Version of TeXCAD: 4.3
% Reference / build: 30-Jun-2012 (rev. 105)
% For new versions, check: http://texcad.sf.net/
% Options on the following lines.
%\grade{\on}
%\emlines{\off}
%\epic{\off}
%\beziermacro{\on}
%\reduce{\on}
%\snapping{\on}
%\pvinsert{% Your \input, \def, etc. here}
%\quality{8.000}
%\graddiff{0.005}
%\snapasp{1}
%\zoom{6.7272}
\unitlength 0.8mm % = 2.845pt
\linethickness{0.4pt}
\ifx\plotpoint\undefined\newsavebox{\plotpoint}\fi % GNUPLOT compatibility
\begin{picture}(129,71)(0,0)
\put(31,69){\circle*{2}}
\put(18,57){\circle{2}}
\put(18,56){\line(0,-1){5}}
\put(18,50){\circle*{2}}
\put(18,49){\line(0,-1){5}}
\put(18,43){\circle{2}}
\put(18,42){\line(0,-1){5}}
\put(18,36){\circle*{2}}
\put(18,35){\line(0,-1){5}}
\put(18,29){\circle{2}}
%\dashline{1}(18,28)(18,22)
\put(17.93,27.93){\line(0,-1){.8571}}
\put(17.93,26.215){\line(0,-1){.8571}}
\put(17.93,24.501){\line(0,-1){.8571}}
\put(17.93,22.787){\line(0,-1){.8571}}
%\end
\put(18,21){\circle{2}}
\put(18,20){\line(0,-1){5}}
\put(18,14){\circle*{2}}
\put(18,13){\line(0,-1){5}}
\put(18,7){\circle{2}}
%\dottedbox(16,27)(4,18)
\multiput(15.93,44.93)(0,-.94737){20}{{\rule{.4pt}{.4pt}}}
\multiput(15.93,26.93)(.8,0){6}{{\rule{.4pt}{.4pt}}}
\multiput(15.93,44.93)(.8,0){6}{{\rule{.4pt}{.4pt}}}
\multiput(19.93,44.93)(0,-.94737){20}{{\rule{.4pt}{.4pt}}}
%\end
%\dottedbox(16,5)(4,18)
\multiput(15.93,22.93)(0,-.94737){20}{{\rule{.4pt}{.4pt}}}
\multiput(15.93,4.93)(.8,0){6}{{\rule{.4pt}{.4pt}}}
\multiput(15.93,22.93)(.8,0){6}{{\rule{.4pt}{.4pt}}}
\multiput(19.93,22.93)(0,-.94737){20}{{\rule{.4pt}{.4pt}}}
%\end
\put(44,57){\circle{2}}
\put(44,56){\line(0,-1){5}}
\put(44,50){\circle*{2}}
\put(44,49){\line(0,-1){5}}
\put(44,43){\circle{2}}
%\dashline{1}(44,42)(44,36)
\put(43.93,41.93){\line(0,-1){.8571}}
\put(43.93,40.215){\line(0,-1){.8571}}
\put(43.93,38.501){\line(0,-1){.8571}}
\put(43.93,36.787){\line(0,-1){.8571}}
%\end
\put(44,35){\circle{2}}
\put(44,34){\line(0,-1){5}}
\put(44,28){\circle*{2}}
\put(44,27){\line(0,-1){5}}
\put(44,21){\circle{2}}
%\dottedbox(42,41)(4,18)
\multiput(41.93,58.93)(0,-.94737){20}{{\rule{.4pt}{.4pt}}}
\multiput(41.93,40.93)(.8,0){6}{{\rule{.4pt}{.4pt}}}
\multiput(41.93,58.93)(.8,0){6}{{\rule{.4pt}{.4pt}}}
\multiput(45.93,58.93)(0,-.94737){20}{{\rule{.4pt}{.4pt}}}
%\end
%\dottedbox(42,19)(4,18)
\multiput(41.93,36.93)(0,-.94737){20}{{\rule{.4pt}{.4pt}}}
\multiput(41.93,18.93)(.8,0){6}{{\rule{.4pt}{.4pt}}}
\multiput(41.93,36.93)(.8,0){6}{{\rule{.4pt}{.4pt}}}
\multiput(45.93,36.93)(0,-.94737){20}{{\rule{.4pt}{.4pt}}}
%\end
\put(26,57){\circle{2}}
\put(26,56){\line(0,-1){5}}
\put(26,50){\circle{2}}
\put(26,49){\line(0,-1){5}}
\put(26,43){\circle*{2}}
\put(26,42){\line(0,-1){5}}
\put(26,36){\circle{2}}
%\dashline{1}(26,35)(26,29)
\put(25.93,34.93){\line(0,-1){.8571}}
\put(25.93,33.215){\line(0,-1){.8571}}
\put(25.93,31.501){\line(0,-1){.8571}}
\put(25.93,29.787){\line(0,-1){.8571}}
%\end
\put(26,28){\circle{2}}
\put(26,27){\line(0,-1){5}}
\put(26,21){\circle*{2}}
\put(26,20){\line(0,-1){5}}
\put(26,14){\circle{2}}
%\dottedbox(24,34)(4,18)
\multiput(23.93,51.93)(0,-.94737){20}{{\rule{.4pt}{.4pt}}}
\multiput(23.93,33.93)(.8,0){6}{{\rule{.4pt}{.4pt}}}
\multiput(23.93,51.93)(.8,0){6}{{\rule{.4pt}{.4pt}}}
\multiput(27.93,51.93)(0,-.94737){20}{{\rule{.4pt}{.4pt}}}
%\end
%\dottedbox(24,12)(4,18)
\multiput(23.93,29.93)(0,-.94737){20}{{\rule{.4pt}{.4pt}}}
\multiput(23.93,11.93)(.8,0){6}{{\rule{.4pt}{.4pt}}}
\multiput(23.93,29.93)(.8,0){6}{{\rule{.4pt}{.4pt}}}
\multiput(27.93,29.93)(0,-.94737){20}{{\rule{.4pt}{.4pt}}}
%\end
\put(36,57){\circle{2}}
\put(36,56){\line(0,-1){5}}
\put(36,50){\circle{2}}
\put(36,49){\line(0,-1){5}}
\put(36,43){\circle*{2}}
\put(36,42){\line(0,-1){5}}
\put(36,36){\circle{2}}
%\dashline{1}(36,35)(36,29)
\put(35.93,34.93){\line(0,-1){.8571}}
\put(35.93,33.215){\line(0,-1){.8571}}
\put(35.93,31.501){\line(0,-1){.8571}}
\put(35.93,29.787){\line(0,-1){.8571}}
%\end
\put(36,28){\circle{2}}
\put(36,27){\line(0,-1){5}}
\put(36,21){\circle*{2}}
\put(36,20){\line(0,-1){5}}
\put(36,14){\circle{2}}
%\dottedbox(34,34)(4,18)
\multiput(33.93,51.93)(0,-.94737){20}{{\rule{.4pt}{.4pt}}}
\multiput(33.93,33.93)(.8,0){6}{{\rule{.4pt}{.4pt}}}
\multiput(33.93,51.93)(.8,0){6}{{\rule{.4pt}{.4pt}}}
\multiput(37.93,51.93)(0,-.94737){20}{{\rule{.4pt}{.4pt}}}
%\end
%\dottedbox(34,12)(4,18)
\multiput(33.93,29.93)(0,-.94737){20}{{\rule{.4pt}{.4pt}}}
\multiput(33.93,11.93)(.8,0){6}{{\rule{.4pt}{.4pt}}}
\multiput(33.93,29.93)(.8,0){6}{{\rule{.4pt}{.4pt}}}
\multiput(37.93,29.93)(0,-.94737){20}{{\rule{.4pt}{.4pt}}}
%\end
\put(7,57){\circle{2}}
\put(7,56){\line(0,-1){5}}
\put(7,50){\circle*{2}}
\put(7,49){\line(0,-1){5}}
\put(7,43){\circle{2}}
\put(7,42){\line(0,-1){5}}
\put(7,36){\circle*{2}}
\put(7,35){\line(0,-1){5}}
\put(7,29){\circle{2}}
%\dashline{1}(7,28)(7,22)
\put(6.93,27.93){\line(0,-1){.8571}}
\put(6.93,26.215){\line(0,-1){.8571}}
\put(6.93,24.501){\line(0,-1){.8571}}
\put(6.93,22.787){\line(0,-1){.8571}}
%\end
\put(7,21){\circle{2}}
\put(7,20){\line(0,-1){5}}
\put(7,14){\circle*{2}}
\put(7,13){\line(0,-1){5}}
\put(7,7){\circle{2}}
%\dottedbox(5,27)(4,18)
\multiput(4.93,44.93)(0,-.94737){20}{{\rule{.4pt}{.4pt}}}
\multiput(4.93,26.93)(.8,0){6}{{\rule{.4pt}{.4pt}}}
\multiput(4.93,44.93)(.8,0){6}{{\rule{.4pt}{.4pt}}}
\multiput(8.93,44.93)(0,-.94737){20}{{\rule{.4pt}{.4pt}}}
%\end
%\dottedbox(5,5)(4,18)
\multiput(4.93,22.93)(0,-.94737){20}{{\rule{.4pt}{.4pt}}}
\multiput(4.93,4.93)(.8,0){6}{{\rule{.4pt}{.4pt}}}
\multiput(4.93,22.93)(.8,0){6}{{\rule{.4pt}{.4pt}}}
\multiput(8.93,22.93)(0,-.94737){20}{{\rule{.4pt}{.4pt}}}
%\end
\put(55,57){\circle{2}}
\put(55,56){\line(0,-1){5}}
\put(55,50){\circle*{2}}
\put(55,49){\line(0,-1){5}}
\put(55,43){\circle{2}}
%\dashline{1}(55,42)(55,36)
\put(54.93,41.93){\line(0,-1){.8571}}
\put(54.93,40.215){\line(0,-1){.8571}}
\put(54.93,38.501){\line(0,-1){.8571}}
\put(54.93,36.787){\line(0,-1){.8571}}
%\end
\put(55,35){\circle{2}}
\put(55,34){\line(0,-1){5}}
\put(55,28){\circle*{2}}
\put(55,27){\line(0,-1){5}}
\put(55,21){\circle{2}}
%\dottedbox(53,41)(4,18)
\multiput(52.93,58.93)(0,-.94737){20}{{\rule{.4pt}{.4pt}}}
\multiput(52.93,40.93)(.8,0){6}{{\rule{.4pt}{.4pt}}}
\multiput(52.93,58.93)(.8,0){6}{{\rule{.4pt}{.4pt}}}
\multiput(56.93,58.93)(0,-.94737){20}{{\rule{.4pt}{.4pt}}}
%\end
%\dottedbox(53,19)(4,18)
\multiput(52.93,36.93)(0,-.94737){20}{{\rule{.4pt}{.4pt}}}
\multiput(52.93,18.93)(.8,0){6}{{\rule{.4pt}{.4pt}}}
\multiput(52.93,36.93)(.8,0){6}{{\rule{.4pt}{.4pt}}}
\multiput(56.93,36.93)(0,-.94737){20}{{\rule{.4pt}{.4pt}}}
%\end
%\emline(31,68)(7,58)
\multiput(31,68)(-.0808080808,-.0336700337){297}{\line(-1,0){.0808080808}}
%\end
%\emline(31,68)(18,58)
\multiput(31,68)(-.0437710438,-.0336700337){297}{\line(-1,0){.0437710438}}
%\end
\put(31,68){\line(-1,-2){5}}
\put(31,68){\line(1,-2){5}}
%\emline(31,68)(44,58)
\multiput(31,68)(.0437710438,-.0336700337){297}{\line(1,0){.0437710438}}
%\end
%\emline(31,68)(55,58)
\multiput(31,68)(.0808080808,-.0336700337){297}{\line(1,0){.0808080808}}
%\end
\put(108,69){\circle{2}}
\put(103,56){\line(0,-1){5}}
\put(103,49){\line(0,-1){5}}
\put(103,43){\circle{2}}
\put(103,42){\line(0,-1){5}}
\put(103,36){\circle*{2}}
\put(103,35){\line(0,-1){5}}
\put(103,29){\circle{2}}
%\dashline{1}(103,28)(103,22)
\put(102.93,27.93){\line(0,-1){.8571}}
\put(102.93,26.215){\line(0,-1){.8571}}
\put(102.93,24.501){\line(0,-1){.8571}}
\put(102.93,22.787){\line(0,-1){.8571}}
%\end
\put(103,21){\circle{2}}
\put(103,20){\line(0,-1){5}}
\put(103,14){\circle*{2}}
\put(103,13){\line(0,-1){5}}
\put(103,7){\circle{2}}
%\dottedbox(101,27)(4,18)
\multiput(100.93,44.93)(0,-.94737){20}{{\rule{.4pt}{.4pt}}}
\multiput(100.93,26.93)(.8,0){6}{{\rule{.4pt}{.4pt}}}
\multiput(100.93,44.93)(.8,0){6}{{\rule{.4pt}{.4pt}}}
\multiput(104.93,44.93)(0,-.94737){20}{{\rule{.4pt}{.4pt}}}
%\end
%\dottedbox(101,5)(4,18)
\multiput(100.93,22.93)(0,-.94737){20}{{\rule{.4pt}{.4pt}}}
\multiput(100.93,4.93)(.8,0){6}{{\rule{.4pt}{.4pt}}}
\multiput(100.93,22.93)(.8,0){6}{{\rule{.4pt}{.4pt}}}
\multiput(104.93,22.93)(0,-.94737){20}{{\rule{.4pt}{.4pt}}}
%\end
\put(89,56){\line(0,-1){5}}
\put(89,49){\line(0,-1){5}}
\put(89,43){\circle{2}}
\put(89,42){\line(0,-1){5}}
\put(89,36){\circle*{2}}
\put(89,35){\line(0,-1){5}}
\put(89,29){\circle{2}}
%\dashline{1}(89,28)(89,22)
\put(88.93,27.93){\line(0,-1){.8571}}
\put(88.93,26.215){\line(0,-1){.8571}}
\put(88.93,24.501){\line(0,-1){.8571}}
\put(88.93,22.787){\line(0,-1){.8571}}
%\end
\put(89,21){\circle{2}}
\put(89,20){\line(0,-1){5}}
\put(89,14){\circle*{2}}
\put(89,13){\line(0,-1){5}}
\put(89,7){\circle{2}}
%\dottedbox(87,27)(4,18)
\multiput(86.93,44.93)(0,-.94737){20}{{\rule{.4pt}{.4pt}}}
\multiput(86.93,26.93)(.8,0){6}{{\rule{.4pt}{.4pt}}}
\multiput(86.93,44.93)(.8,0){6}{{\rule{.4pt}{.4pt}}}
\multiput(90.93,44.93)(0,-.94737){20}{{\rule{.4pt}{.4pt}}}
%\end
%\dottedbox(87,5)(4,18)
\multiput(86.93,22.93)(0,-.94737){20}{{\rule{.4pt}{.4pt}}}
\multiput(86.93,4.93)(.8,0){6}{{\rule{.4pt}{.4pt}}}
\multiput(86.93,22.93)(.8,0){6}{{\rule{.4pt}{.4pt}}}
\multiput(90.93,22.93)(0,-.94737){20}{{\rule{.4pt}{.4pt}}}
%\end
\put(108,68){\line(-1,-2){5}}
%\emline(108,68)(89,58)
\multiput(108,68)(-.063973064,-.0336700337){297}{\line(-1,0){.063973064}}
%\end
\put(108,68){\line(1,-2){5}}
%\emline(108,68)(127,58)
\multiput(108,68)(.063973064,-.0336700337){297}{\line(1,0){.063973064}}
%\end
\put(113,57){\circle{2}}
\put(113,56){\line(0,-1){5}}
\put(113,50){\circle*{2}}
\put(113,49){\line(0,-1){5}}
\put(113,43){\circle{2}}
%\dashline{1}(113,42)(113,36)
\put(112.93,41.93){\line(0,-1){.8571}}
\put(112.93,40.215){\line(0,-1){.8571}}
\put(112.93,38.501){\line(0,-1){.8571}}
\put(112.93,36.787){\line(0,-1){.8571}}
%\end
\put(113,35){\circle{2}}
\put(113,34){\line(0,-1){5}}
\put(113,28){\circle*{2}}
\put(113,27){\line(0,-1){5}}
\put(113,21){\circle{2}}
%\dottedbox(111,41)(4,18)
\multiput(110.93,58.93)(0,-.94737){20}{{\rule{.4pt}{.4pt}}}
\multiput(110.93,40.93)(.8,0){6}{{\rule{.4pt}{.4pt}}}
\multiput(110.93,58.93)(.8,0){6}{{\rule{.4pt}{.4pt}}}
\multiput(114.93,58.93)(0,-.94737){20}{{\rule{.4pt}{.4pt}}}
%\end
%\dottedbox(111,19)(4,18)
\multiput(110.93,36.93)(0,-.94737){20}{{\rule{.4pt}{.4pt}}}
\multiput(110.93,18.93)(.8,0){6}{{\rule{.4pt}{.4pt}}}
\multiput(110.93,36.93)(.8,0){6}{{\rule{.4pt}{.4pt}}}
\multiput(114.93,36.93)(0,-.94737){20}{{\rule{.4pt}{.4pt}}}
%\end
\put(127,57){\circle{2}}
\put(127,56){\line(0,-1){5}}
\put(127,50){\circle*{2}}
\put(127,49){\line(0,-1){5}}
\put(127,43){\circle{2}}
%\dashline{1}(127,42)(127,36)
\put(126.93,41.93){\line(0,-1){.8571}}
\put(126.93,40.215){\line(0,-1){.8571}}
\put(126.93,38.501){\line(0,-1){.8571}}
\put(126.93,36.787){\line(0,-1){.8571}}
%\end
\put(127,35){\circle{2}}
\put(127,34){\line(0,-1){5}}
\put(127,28){\circle*{2}}
\put(127,27){\line(0,-1){5}}
\put(127,21){\circle{2}}
%\dottedbox(125,41)(4,18)
\multiput(124.93,58.93)(0,-.94737){20}{{\rule{.4pt}{.4pt}}}
\multiput(124.93,40.93)(.8,0){6}{{\rule{.4pt}{.4pt}}}
\multiput(124.93,58.93)(.8,0){6}{{\rule{.4pt}{.4pt}}}
\multiput(128.93,58.93)(0,-.94737){20}{{\rule{.4pt}{.4pt}}}
%\end
%\dottedbox(125,19)(4,18)
\multiput(124.93,36.93)(0,-.94737){20}{{\rule{.4pt}{.4pt}}}
\multiput(124.93,18.93)(.8,0){6}{{\rule{.4pt}{.4pt}}}
\multiput(124.93,36.93)(.8,0){6}{{\rule{.4pt}{.4pt}}}
\multiput(128.93,36.93)(0,-.94737){20}{{\rule{.4pt}{.4pt}}}
%\end
\put(89,57){\circle*{2}}
\put(103,57){\circle*{2}}
\put(89,50){\circle{2}}
\put(103,50){\circle{2}}
\put(0,0){$P_{3a_1\!+\!2}$}
\put(13,0){$P_{3a_s\!+\!2}$}
\put(20,8){$P_{3b_1\!+\!1}$}
\put(33,8){$P_{3b_t\!+\!1}$}
\put(40,15){$P_{3c_1}$}
\put(53,15){$P_{3c_h}$}
\put(84,0){$P_{3a_1\!+\!2}$}
\put(98,0){$P_{3a_s\!+\!2}$}
\put(110,15){$P_{3c_1}$}
\put(123,15){$P_{3c_h}$}
\put(118,55){$\ldots$}
\put(94,55){$\ldots$}
\put(10,55){$\ldots$}
\put(28,55){$\ldots$}
\put(47,55){$\ldots$}
\put(30,71){$u$}
\put(107,71){$u$}
\end{picture}

  \caption{\footnotesize The starlike tree with the minimum dominating set. }\label{starlike}
\end{figure}

It is clear that the following result holds by employing  Fig.\ref{starlike}.

\begin{lem}\label{starlike-tree-domination-number}
Let $T$ be a starlike tree on $n$ vertices with $T-u=
P_{3a_1+2}\cup \cdots \cup P_{3a_s+2}
\cup P_{3b_1+1} \cup \cdots \cup P_{3b_t+1}
\cup P_{3c_1}\cup \cdots \cup P_{3c_h}
$, where $u$ is the center vertex of $T$ and $s,t,h\ge0$.
Then
$$\gamma(T)=
\left\{\begin{array}{ll}
\frac{n+s-t+2}{3},& \mbox{if $t \ge 1$,}\\
\frac{n+s-1}{3},& \mbox{if $t=0$.}
\end{array}
\right.$$
\end{lem}

\begin{lem}[\cite{Cvetkovic-Rowlinson-Simic-2010}]\label{qubian}
For any edge $uv$ of a tree $T$,
we have $f(T,x)=f(T-uv,x)-f(T-u-v,x)$.
\end{lem}

%\textcolor[rgb]{1.00,0.00,0.00}{If $h(x)$ is a polynomial in the variable $x$, let $\rho(h)$ denote the largest real root of equation $h(x) = 0$.
% \begin{lem}[\cite{Li-Shi-2010}]\label{compare}
% Let $g(x), h(x)$ be monic polynomials with real roots. If $h(x) < g(x)$ for all $x\ge \rho(g)$, then
%$\rho(h)>\rho(g)$.
%\end{lem}}

\begin{lem}[\cite{Wang-Huang-2010}]\label{compare}
Let $G$ and $H$ be two graphs with the characteristic polynomials $f(G,x)$ and $f(H,x)$, respectively.
If $f(G,x) > f(H,x )$ for any $x \ge \rho(G)$, then $\rho(G) < \rho(H )$.
\end{lem}

\section{Proof}
In this section, we always assume that  $n$ is odd and
$T^*$ is the minimizer graph in $\mathbb{G}_{n,\lfloor\frac{n}{2}\rfloor}$.
Clearly, $\gamma(T^*)=\frac{n-1}{2}$. By Lemma \ref{min-tree}, we know that $T^*$ is a tree.
For any vertex $u$ of a tree $T$,
let  $L_{T}(u)$ be the set of leaves adjacent to vertex $u$ in $T$ and denote by $l_{T}(u)=|L_{T}(u)|$. Particularly, if $l_{T}(u)=0$ then there is no leaves adjacent to $u$ in $T$.
For a minimum domination set $D$ of $T$, let $v\in D$ and $u\in N_T(v)$.
If there is no vertex in $D\setminus \{v\}$ dominating $u$, then we call $u$ is \emph{uniquely dominated} by $v$ relative to $D$.
Firstly, we will give an upper bound of the spectral radius of $T^*$.

\begin{lem}\label{upbound}
For odd $n \ge 3$,
let $T^*$ be the minimizer graph in $ \mathbb{G}_{n,\lfloor\frac{n}{2}\rfloor}$.
Then $\rho(T^*)< 1+\sqrt{2}$.
\end{lem}

\begin{proof}
For any $i\in [\lceil\frac{n-3}{4}\rceil,\frac{n-3}{2}]$,
recall that $T_{i,\frac{n-3}{2}-i}^{\frac{n+3}{2}} \in \mathbb{G}_{n,\lfloor\frac{n}{2}\rfloor}$ and
$T_{i,\frac{n-3}{2}-i}^{\frac{n+3}{2}}$ (see Fig.\ref{T}) is obtained from $P_{\frac{n+3}{2}}\circ K_1$ by deleting three leaves adjacent to $v_{i+1}, v_{i+2}, v_{i+3}$ of the path  $P_{\frac{n+3}{2}}=v_1v_2\cdots v_{\frac{n+3}{2}}$.
Thus, $\rho(T^*)\le \rho(T_{i,\frac{n-3}{2}-i}^{\frac{n+3}{2}})< \rho(P_{\frac{n+3}{2}}\circ K_1)$.
By Lemma \ref{circ} and the spectral radius of the path $P_{\frac{n+3}{2}}$, we have
$$
\rho(T^*)
\!<\! \rho(P_{\frac{n+3}{2}}\circ K_1)
\!=\!\frac{1}{2}(\rho(P_{\frac{n+3}{2}})+\sqrt{\rho^2(P_{\frac{n+3}{2}})+4})
\!=\!\cos\frac{2\pi}{n+5}+\sqrt{\cos^2\frac{2\pi}{n+5} +1}
\!<\!1+\sqrt{2}.
$$
It completes the proof.
\end{proof}

\begin{lem}\label{leave-number}
For odd $n \ge 5$,
let $T^*$ be the minimizer graph in $\mathbb{G}_{n,\lfloor\frac{n}{2}\rfloor}$.
Then $l_{T^*}(u)\le 1$ for any $u\in V(T^*)$.
\end{lem}

\begin{proof}
Firstly, suppose that there exists a vertex $u\in V(T^*)$ such that $l_{T^*}(u)\geq 3$.
By  Lemma \ref{dominating-set-support}, we may denote
 $D^*$  a  minimum dominating set of $T^*$ containing all support vertices.
Clearly, $u\in D^*$ and
all  leaves of $u$  are uniquely dominated by $u$  relative to $D^*$.
Let $T'=T^*- (L_{T^*}(u)\cup \{u\})$.
By Lemma \ref{domination-number-upbound}, we have
$$
\gamma(T')\leq \frac{n(T')}{2}= \frac{n-l_{T^*}(u)-1}{2}\le \frac{n-3-1}{2}= \frac{n-4}{2}.
$$
Therefore,
$$
\gamma(T^*)=\gamma(T')+1\le \frac{n-4}{2}+1< \frac{n-1}{2},
$$
a contradiction.
Next we
may suppose that there exists a vertex $u\in V(T^*)$ such that  $l_{T^*}(u)=2$.
Denote  $L_{T^*}(u)=\{v_1, v_2\}$.
It is clear that $v_1$ and $v_2$ are uniquely dominated by $u$ relative to  $D^*$.
Since $\gamma(T^*)=\frac{n-1}{2}$,
all vertices in  $ N_{T^*}(u)\setminus \{v_1, v_2\}$ are not uniquely dominated by $u$ relative to  $D^*$.
Now we construct a new tree $T$ obtained from $T^*$ by deleting  $v_2$ and subdivision the edge $uv_1$ once.
Note that $n\ge 5$.  We get that $T\not\cong T^*$, and
$D=D^*-\{u\}+\{v_2\}$ is a minimum dominating set of $T$. Thus, $T\in \mathbb{G}_{n,\lfloor\frac{n}{2}\rfloor}$.
By Lemma \ref{lahuan}, we have $\rho(T)< \rho(T^*)$, which is a contradiction.
Therefore,  $l_{T^*}(u)\le 1$ for any $u\in V(T^*)$.
\end{proof}

%\textcolor[rgb]{0.44,0.00,0.94}{We call  \emph{a pendant $3$-path} of a graph $G$ if $G$ contains an induced $3$-path, say $P_3:ww'w''$,   such that
%$d_{T^*}(w)=1$, $d_{T^*}(w')=2$  and $d_{G}(w'')\ge 2$.}

If $G$ contains an induced $3$-path, say $P_3:ww'w''$,   with
$d_{T^*}(w)=1$, $d_{T^*}(w')=2$  and $d_{G}(w'')\ge 2$,
then we call $P_3:ww'w''$ is a \emph{pendant $3$-path} of a graph $G$.
Since $T^*$ is a tree, from  Lemma \ref{leave-number},
every diametrical  path of $T^*$ contains  two pendant $3$-paths.
Thus, $T^*$ contains at least two pendant $3$-paths.

\begin{lem}\label{P3}
For odd $n \ge 5$,
let $T^*$ be the minimizer graph in $\mathbb{G}_{n,\lfloor\frac{n}{2}\rfloor}$ and $D^*$ be a  minimum dominating set of $T^*$ containing all support vertices.
Then  $T^*$ contains at most one pendant $3$-path, say $P_3:ww'w''$, with
$d_{T^*}(w)=1$, $d_{T^*}(w')=2$ and $d_{T^*}(w'')\ge 2$
such that $w''$ is  uniquely dominated by $w'$ relative to  $D^*$.
   \end{lem}
 \begin{proof}
Suppose to the contrary  that $T^*$ contains   two pendant $3$-paths, say  $P_3^1:w_1 w_1' w''_1$ and $P_3^2:w_2 w_2' w''_2$, with $d_{T^*}(w_1)=d_{T^*}(w_2)=1$, $d_{T^*}(w'_1)=d_{T^*}(w'_2)=2$ and  $d_{T^*}(w''_1)$, $d_{T^*}(w''_2)\ge 2$
such that  $w'_1,w'_2\in D^*$ and
$w''_1,w''_2$ are uniquely dominated by  $w'_1$, $w'_2$ relative to  $D^*$, respectively.
Then we have $\gamma(T^*)=\gamma(T^*- \{P^1_3, P^2_3\})+2$. Notice that $n(T^*- \{P^1_3, P^2_3\})=n-6$.
By Lemma \ref{domination-number-upbound} we have
$$
\gamma(T^*)=\gamma(T^*- \{P^1_3, P^2_3\})+2
\le \frac{n-6}{2}+2<\frac{n-1}{2},
$$
a contradiction.
It completes the proof.
 \end{proof}

\begin{lem}\label{P3-exist}
For odd $n \ge 5$, let $T^*$ be the minimizer graph in $\mathbb{G}_{n,\lfloor\frac{n}{2}\rfloor}$.
For any  pendant $3$-path $P_3: ww'w''$ with $d_{T^*}(w)=1$, $d_{T^*}(w')=2$ and $d_{T^*}(w'')\ge 2$,
there exists a minimum dominating set $D$ of $T^*$ such that $w\notin D$, $w'\in D$ and
$w''$ is not uniquely dominated by $w'$ relative to  $D$.
   \end{lem}
\begin{proof}
Let  $D^*$ be a  minimum dominating set of $T^*$ containing all support vertices.
For any  pendant $3$-path $P_3: ww'w''$ with $d_{T^*}(w)=1$, $d_{T^*}(w')=2$ and $d_{T^*}(w'')\ge 2$,
if $w''$ is  not uniquely dominated by $w'$ relative to  $D^*$, then the result immediately holds.
If $w''$ is  uniquely dominated by $w'$ relative to  $D^*$, then we have
 $$\gamma(T^*-P_3)=\gamma(T^*)-1=\frac{n-3}{2}
 =\frac{n(T^*-P_3)}{2}$$ due to $\gamma(T^*)=\frac{n-1}{2}$.
 Since $n(T^*-P_3)=n-3$ is even,
 by Lemma \ref{half-domination} we have
 $B=V(T^*-P_3)\setminus D^*$ is also a minimum dominating set of
 $T^*-P_3$.
 Take $D=B\cup \{w'\}$.
 Clearly,  $D$ is a   minimum dominating set of
 $T^*$ and $w\notin D$, $w'\in N_{T^*}(w'')\subseteq D$.
 Thus, $w''$ is not uniquely dominated by $w'$ relative to  $D$.
\end{proof}

\begin{lem}\label{max-degree-less-than4}
For odd $n \ge 11$,
let $T^*$ be the minimizer graph in $\mathbb{G}_{n,\lfloor\frac{n}{2}\rfloor}$.
Then $3\le \Delta(T^*)\le 4$.
\end{lem}
\begin{proof}
Since $T^*$ is a tree, we have $\Delta(T^*)\ge 2$.
If $\Delta(T^*)=2$ then $T^*\cong P_n$.
Note that $n \ge 11$, we have $\gamma(T^*)=\gamma(P_n)=\lceil\frac{n}{3}\rceil< \lfloor\frac{n}{2}\rfloor$, a contradiction.
Thus,  $\Delta(T^*)\ge 3$.

Suppose to the contrary that $\Delta(T^*)\ge 5$.
%Assume that
% $\Delta(T^*)\ge 5$.
From Lemma \ref{leave-number},
$S_{10}$ (see Fig.\ref{W-S}) is a proper subgraph of $T^*$.
%A simple calculation shows $\rho(S_8)=1+\sqrt{2}$,
By Lemma \ref{subgraph}, we have
$\rho(T^*)>\rho(S_{10})=1+\sqrt{2}$, which contradicts Lemma \ref{upbound}.
Thus,  $\Delta(T^*)\leq4$.
\end{proof}

\begin{lem}\label{branching-vertex}
For odd $n \ge 13$,
let $T^*$ be the minimizer graph in $\mathbb{G}_{n,\lfloor\frac{n}{2}\rfloor}$.
Then $T^*$ contains at least two branching vertices.
\end{lem}
\begin{proof}
By Lemma \ref{max-degree-less-than4}, $T^*$ contains at least one branching vertex.
Suppose to the contrary that
$T^*$ contains exactly one branching vertex,
i.e.,  $T^*$ is a starlike tree.
Let
$T^*-u=
P_{3a_1+2}\cup \cdots \cup P_{3a_s+2}
\cup P_{3b_1+1} \cup \cdots \cup P_{3b_t+1}
\cup P_{3c_1}\cup \cdots \cup P_{3c_h}$,
where $u$ is the center vertex of $T^*$ and $s,t,h\ge0$.
By Lemma \ref{max-degree-less-than4}, we have $s+t+h\le 4$,
and so $s\le 4$, $s+t\le 4$.
 Since $n\ge13$, from Lemma \ref{starlike-tree-domination-number} we have
$$\gamma(T^*)=\left\{  \begin{array}{ll}
\frac{n+s-1}{3}\le \frac{n+3}{3}< \frac{n-1}{2},& \mbox{ if $t=0$,}\vspace{0.1cm}\\
\frac{n+s-t+2}{3}=\frac{n+s+t-2t+2}{3}\le \frac{n+4-2+2}{3}< \frac{n-1}{2},& \mbox{ if $t\ge 1$,}
\end{array}\right.$$
 which contradicts  $\gamma(T^*)=\frac{n-1}{2}$.
It completes the proof.
\end{proof}

\begin{lem}\label{max-degree}
For odd $n \ge 13$,
let $T^*$ be the minimizer graph in $\mathbb{G}_{n,\lfloor\frac{n}{2}\rfloor}$.
Then $\Delta(T^*)=3$.
\end{lem}

\begin{proof}
It is sufficient to prove $\Delta(T^*)\neq 4$ from Lemma \ref{max-degree-less-than4}.
Suppose  that $u$ is a  vertex of $T^*$ satisfying $d_{T^*}(u)=\Delta(T^*)=4$.
Let  $N_{T^*}(u)= \{v_1, v_2, v_3, v_4\}$
and
$T^*-u=\cup_{i=1}^4 T_{v_i}$, where $T_{v_i}$ is the component containing vertex $v_i$ for $1\le i\le 4$.
Let $D^*$ be  a minimum dominating set of $T^*$ containing all support vertices.
By Lemma \ref{leave-number} we have $l_{T^*}(u)\le 1$.
Next we will distinguish the following two cases to lead a contradiction.

{\flushleft\bf Case 1 } $l_{T^*}(u)= 1$.

We may denote $L_{T^*}(u)= \{v_4\}$.
Clearly, $u\in D^*$ and
$v_4$ is uniquely dominated by $u$ relative to   $D^*$.
Recall that $\gamma(T^*)=\frac{n-1}{2}$.
Then there exist at most two vertices uniquely dominated by $u$ relative to   $D^*$.

{\flushleft\bf Subcase 1.1 }
Assume that  one of $v_1, v_2, v_3$ is uniquely dominated by $u$ relative to   $D^*$.
Without loss of generality assume that $v_1$ is  such vertex. Clearly, $v_1\notin D^*$.
Let $G_1$ be the tree obtained from $T^*$ by subdividing the edge $uv_1$ once,
say $u'$ the new vertex,
and deleting the leaf $v_4$ (see Fig.\ref{case1}).
Denote $D_1=D^*-\{u\}+\{u'\}$.
It is easy to verify that $D_1$ is a minimum dominating set of $G_1$ and $|D_1|=|D^*|$.
Thus, $G_1 \in \mathbb{G}_{n,\lfloor\frac{n}{2}\rfloor}$.
If the component $T_{v_1}$ is a path, by Lemma \ref{lahuan}
then  $\rho(G_1)< \rho(T^*)$, a contradiction.
If $T_{v_1}$ is not a path, then
$uv_1$ is an edge on some internal path of $T^*$.
By Lemmas \ref{subdivision} and \ref{subgraph},
we have $\rho(G_1)< \rho(T^*)$, a contradiction.

\begin{figure}[h]
  \centering
    \footnotesize
% This is a LaTeX picture output by TeXCAD.
% This is a LaTeX picture output by TeXCAD.
% File name: [Clipboard].
% Version of TeXCAD: 4.3
% Reference / build: 30-Jun-2012 (rev. 105)
% For new versions, check: http://texcad.sf.net/
% Options on the following lines.
%\grade{\on}
%\emlines{\off}
%\epic{\off}
%\beziermacro{\on}
%\reduce{\on}
%\snapping{\on}
%\pvinsert{% Your \input, \def, etc. here}
%\quality{8.000}
%\graddiff{0.005}
%\snapasp{1}
%\zoom{5.6569}
\unitlength 0.8mm % = 2.845pt
\linethickness{0.4pt}
\ifx\plotpoint\undefined\newsavebox{\plotpoint}\fi % GNUPLOT compatibility
\begin{picture}(193,25)(0,0)
\put(144,23){\circle*{2}}
\put(144,23){\line(-5,-3){15}}
\put(129,13){\circle{2}}
\put(129,12){\line(-2,-5){2}}
\put(129,12){\line(2,-5){2}}
%\emline(144,23)(139,14)
\multiput(144,23)(-.033557047,-.060402685){149}{\line(0,-1){.060402685}}
%\end
%\emline(144,23)(148,14)
\multiput(144,23)(.033613445,-.075630252){119}{\line(0,-1){.075630252}}
%\end
\put(144,23){\line(4,-3){12}}
\put(156,13){\circle{2}}
\put(139,12){\line(-2,-5){2}}
\put(139,12){\line(2,-5){2}}
\put(148,12){\line(-2,-5){2}}
\put(148,12){\line(2,-5){2}}
\put(144,25){$u$}
\put(124,13){$v_1$}
\put(134,13){$v_2$}
\put(143,13){$v_3$}
\put(151,13){$v_4$}
\put(172,13){\circle{2}}
\put(172,12){\line(-2,-5){2}}
\put(172,12){\line(2,-5){2}}
\put(182,12){\line(-2,-5){2}}
\put(182,12){\line(2,-5){2}}
\put(191,12){\line(-2,-5){2}}
\put(191,12){\line(2,-5){2}}
\put(187,25){$u$}
\put(175,18){$u'$}
\put(167,13){$v_1$}
\put(177,13){$v_2$}
\put(186,13){$v_3$}
\put(187,23){\circle{2}}
\put(187,22){\line(1,-2){4}}
%\emline(187,22)(182,14)
\multiput(187,22)(-.033557047,-.053691275){149}{\line(0,-1){.053691275}}
%\end
\put(187,22){\line(-2,-1){8}}
\put(179,18){\circle*{2}}
%\emline(179,18)(172,14)
\multiput(179,18)(-.058823529,-.033613445){119}{\line(-1,0){.058823529}}
%\end
\put(160,13){\vector(1,0){5}}
\put(122,0){$T^*$\!\! (\!the \!dominating \!set\! $D^*_1$\!)}
\put(178,0){$G_1$}
\put(139,13){\circle*{2}}
\put(148,13){\circle*{2}}
\put(182,13){\circle*{2}}
\put(191,13){\circle*{2}}
\put(102,23){\circle*{2}}
\put(102,23){\line(-5,-3){15}}
\put(87,12){\line(-2,-5){2}}
\put(87,12){\line(2,-5){2}}
%\emline(102,23)(97,14)
\multiput(102,23)(-.033557047,-.060402685){149}{\line(0,-1){.060402685}}
%\end
%\emline(102,23)(106,14)
\multiput(102,23)(.033613445,-.075630252){119}{\line(0,-1){.075630252}}
%\end
\put(102,23){\line(4,-3){12}}
\put(114,13){\circle{2}}
\put(97,12){\line(-2,-5){2}}
\put(97,12){\line(2,-5){2}}
\put(106,12){\line(-2,-5){2}}
\put(106,12){\line(2,-5){2}}
\put(102,25){$u$}
\put(82,13){$v_1$}
\put(92,13){$v_2$}
\put(101,13){$v_3$}
\put(109,13){$v_4$}
\put(118,13){\vector(1,0){5}}
\put(78,0){$T^*$ in Subcase 1.2(ii)}
\put(97,13){\circle*{2}}
\put(106,13){\circle*{2}}
\put(87,13){\circle*{2}}
%\dottedbox(126,5)(6,10)
\multiput(125.93,14.93)(0,-.90909){12}{{\rule{.4pt}{.4pt}}}
\multiput(125.93,4.93)(.85714,0){8}{{\rule{.4pt}{.4pt}}}
\multiput(125.93,14.93)(.85714,0){8}{{\rule{.4pt}{.4pt}}}
\multiput(131.93,14.93)(0,-.90909){12}{{\rule{.4pt}{.4pt}}}
%\end
%\dottedbox(84,5)(6,10)
\multiput(83.93,14.93)(0,-.90909){12}{{\rule{.4pt}{.4pt}}}
\multiput(83.93,4.93)(.85714,0){8}{{\rule{.4pt}{.4pt}}}
\multiput(83.93,14.93)(.85714,0){8}{{\rule{.4pt}{.4pt}}}
\multiput(89.93,14.93)(0,-.90909){12}{{\rule{.4pt}{.4pt}}}
%\end
\put(20,23){\circle*{2}}
\put(20,23){\line(-5,-3){15}}
\put(5,13){\circle{2}}
\put(5,12){\line(-2,-5){2}}
\put(5,12){\line(2,-5){2}}
%\emline(20,23)(15,14)
\multiput(20,23)(-.033557047,-.060402685){149}{\line(0,-1){.060402685}}
%\end
%\emline(20,23)(24,14)
\multiput(20,23)(.033613445,-.075630252){119}{\line(0,-1){.075630252}}
%\end
\put(20,23){\line(4,-3){12}}
\put(32,13){\circle{2}}
\put(15,13){\circle{2}}
\put(15,12){\line(-2,-5){2}}
\put(15,12){\line(2,-5){2}}
\put(24,13){\circle{2}}
\put(24,12){\line(-2,-5){2}}
\put(24,12){\line(2,-5){2}}
\put(14,13){\line(1,0){2}}
\put(15,14){\line(0,-1){2}}
\put(23,13){\line(1,0){2}}
\put(24,14){\line(0,-1){2}}
\put(20,25){$u$}
\put(0,13){$v_1$}
\put(10,13){$v_2$}
\put(19,13){$v_3$}
\put(27,13){$v_4$}
\put(48,13){\circle{2}}
\put(48,12){\line(-2,-5){2}}
\put(48,12){\line(2,-5){2}}
\put(58,13){\circle{2}}
\put(58,12){\line(-2,-5){2}}
\put(58,12){\line(2,-5){2}}
\put(67,13){\circle{2}}
\put(67,12){\line(-2,-5){2}}
\put(67,12){\line(2,-5){2}}
\put(57,13){\line(1,0){2}}
\put(58,14){\line(0,-1){2}}
\put(66,13){\line(1,0){2}}
\put(67,14){\line(0,-1){2}}
\put(63,25){$u$}
\put(51,18){$u'$}
\put(43,13){$v_1$}
\put(53,13){$v_2$}
\put(62,13){$v_3$}
\put(63,23){\circle{2}}
\put(63,22){\line(1,-2){4}}
%\emline(63,22)(58,14)
\multiput(63,22)(-.033557047,-.053691275){149}{\line(0,-1){.053691275}}
%\end
\put(63,22){\line(-2,-1){8}}
\put(55,18){\circle*{2}}
%\emline(55,18)(48,14)
\multiput(55,18)(-.058823529,-.033613445){119}{\line(-1,0){.058823529}}
%\end
\put(36,13){\vector(1,0){5}}
\put(-2,0){$T^*$ in Subcase 1.1(1.2(i))}
\put(54,0){$G_1$}
\end{picture}
  \caption{ \footnotesize Graphs $T^*$ and $G_1$ in Case 1 (the black vertex represents the dominating vertex, the white vertex represents the non-dominating vertex, and the other vertex may be the dominating vertex or not). }\label{case1}
\end{figure}
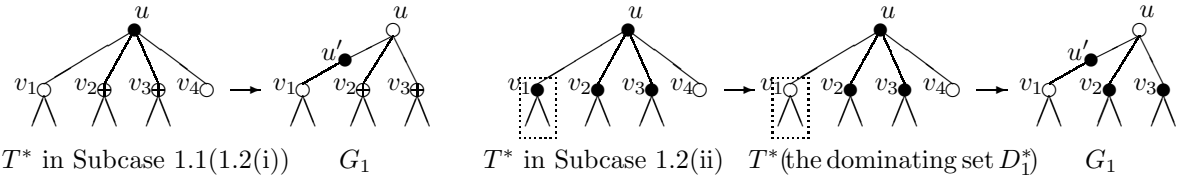
{\flushleft\bf Subcase 1.2 }
Assume that  $v_1, v_2, v_3$ are not uniquely dominated by $u$ relative to    $D^*$.
Then  $v_1, v_2, v_3$ belong to $D^*$ or not.

(i) Assume that  there exists one vertex, say $v_1$,  in $\{v_1, v_2, v_3\}$ such that $v_1\notin D^*$.
As similar as the Subcase 1.1,  we can construct the tree $G_1$ (see Fig.\ref{case1}). Recall that $G_1 \in \mathbb{G}_{n,\lfloor\frac{n}{2}\rfloor}$ and
 $\rho(G_1)< \rho(T^*)$, a contradiction.

(ii) Assume that $v_1, v_2, v_3\in D^*$.
Let $n(T_{v_i})=n_i$ for $1\le i\le 4$. Clearly $n=2+n_1+n_2+n_3$.
We claim that
at most one of $n_1,n_2,n_3$ is an odd integer,
since otherwise $\gamma(T_{v_i})\le \frac{n_i-1}{2}$  ($i=1,2,3$) from Lemma \ref{domination-number-upbound}, and so $\gamma(T^*)=\frac{n_1-1}{2}+\frac{n_2-1}{2}
+\frac{n_3-1}{2}+1=\frac{n-3}{2}<\frac{n-1}{2}$, a contradiction.
Without loss of generality assume that $n_1$ is even.
Since   $u, v_1\in D^*$ and $\gamma(T^*)=\frac{n-1}{2}$,
 $V(T_{v_1})\cap D^* $ is a
minimum dominating set of $T_{v_1}$ and $\gamma(T_{v_1})=\frac{n_1}{2}$.
By Lemma \ref{half-domination},
$V(T_{v_1}) \setminus D^*$ is also a
minimum dominating set of $T_{v_1}$.
Let the vertex set
$$D^*_1= [V(T_{v_1}) \setminus D^*]\bigcup
[V(T^*-T_{v_1})\cap D^*]
\mbox{ (see Fig.\ref{case1}).}$$
It is clear that $D^*_1$ is a minimum dominating set of $T^*$ (note that $D^*_1$ does not contain all support vertices) and $|D^*_1|=\gamma(T^*)$.
One can see  $u\in D^*_1$ and $v_1\notin D^*_1$.
Now we can also construct the tree $G_1$ obtained from $T^*$ by subdividing the edge $uv_1$ once,
say $u'$ the new vertex,
and deleting the leaf $v_4$ (see Fig.\ref{case1}).
In this situation,
$D'_1=D^*_1-\{u\}+\{u'\}$ is a minimum dominating set of $G_1$  and $|D'_1|=\gamma(T^*)$.
Thus,  $G_1\in \mathbb{G}_{n,\lfloor\frac{n}{2}\rfloor}$ and  $\rho(G_1)< \rho(T^*)$, which is a contradiction.

{\flushleft\bf Case 2 } $l_{T^*}(u)= 0$.

%Then $d_(v_i)\geq 2$ for $i=1,2,3,4$.
By Lemma \ref{branching-vertex},  we may assume that
$uv_1$ is an edge on some internal path of $T^*$.
One can see that $u$, $v_1$ may belong to $D^*$ or not.
Next we will discuss the following three cases.

{\flushleft\bf Subcase 2.1 } $u\in D^*$ and $v_1\in D^*$.

In this subcase,  $V(T_{v_1})\cap D^*$ and
$V(T^*-T_{v_1})\cap D^*$ is a minimum dominating set of
$T_{v_1}$ and $T^*-T_{v_1}$, respectively. Denote by $n_0=n(T^*-T_{v_1})$ and recall that $n_1=n(T_{v_1})$.

(i)
If $n_0$ is even, then $n_1=n-n_0$ is odd.
By Lemma \ref{domination-number-upbound}, we have $\gamma(T_{v_1})\leq\frac{n_1-1}{2}$ and $\gamma(T^*-T_{v_1})\leq \frac{n_0}{2}$.
Note that $\gamma(T^*)=\frac{n-1}{2}=\frac{n_0+n_1-1}{2}$.
This implies that $\gamma(T^*-T_{v_1})=\frac{n_0}{2}$.
By Lemma \ref{half-domination}, we have
$V(T^*-T_{v_1})\setminus D^*$ is a minimum dominating set
of $T^*-T_{v_1}$,
and so
$$D^*_2= [V(T_{v_1})\cap D^*]\bigcup [V(T^*-T_{v_1})\setminus D^*] \mbox{ (see Fig.\ref{case2-1i})}$$
is a minimum dominating set
of $T^*$.
Since $uv_1$ is an edge on some internal path of $T^*$, the component $T_{v_1}$ has a vertex of degree at least $3$ in $T^*$.
Then $T_{v_1}$ contains at least two pendant $3$-paths.
From  Lemma \ref{P3} there exists at least one pendant $3$-path, say $w w'w''$, of $T^*$
with  $d_{T^*}(w)=1$, $d_{T^*}(w')=2$ and $d_{T^*}(w'')\ge2$  such that $w''$ is not uniquely dominated by $w'$ relative to $D^*$ (or $D^*_2$).
Let $G_2$ be the tree obtained from $T^*$ by
subdividing the edge $uv_1$ twice (say $w_1, w_2$ the two inserting vertices), and deleting two vertices $w$ and $w'$ (see Fig.\ref{case2-1i}).
Clearly,
$D_2=D^*_2-\{w'\}+\{w_1\}$ is a minimum dominating set of $G_2$.
Thus, $G_2\in \mathbb{G}_{n,\lfloor\frac{n}{2}\rfloor}$.
By Lemmas \ref{subdivision} and \ref{subgraph},
we have $\rho(G_2)< \rho(T^*)$, a contradiction.
\begin{figure}[h]
  \centering
  \footnotesize
  % This is a LaTeX picture output by TeXCAD.
% File name: [Clipboard].
% Version of TeXCAD: 4.3
% Reference / build: 30-Jun-2012 (rev. 105)
% For new versions, check: http://texcad.sf.net/
% Options on the following lines.
%\grade{\on}
%\emlines{\off}
%\epic{\off}
%\beziermacro{\on}
%\reduce{\on}
%\snapping{\on}
%\pvinsert{% Your \input, \def, etc. here}
%\quality{8.000}
%\graddiff{0.005}
%\snapasp{1}
%\zoom{6.7272}
\unitlength 0.9mm % = 2.845pt
\linethickness{0.4pt}
\ifx\plotpoint\undefined\newsavebox{\plotpoint}\fi % GNUPLOT compatibility
\begin{picture}(156,42)(0,0)
\put(5,29){\line(-2,-5){2}}
\put(5,29){\line(2,-5){2}}
\put(15,30){\circle{2}}
\put(15,29){\line(-2,-5){2}}
\put(15,29){\line(2,-5){2}}
\put(24,30){\circle{2}}
\put(24,29){\line(-2,-5){2}}
\put(24,29){\line(2,-5){2}}
\put(14,30){\line(1,0){2}}
\put(15,31){\line(0,-1){2}}
\put(23,30){\line(1,0){2}}
\put(24,31){\line(0,-1){2}}
\put(20,42){$u$}
\put(0,30){$v_1$}
\put(10,30){$v_2$}
\put(19,30){$v_3$}
\put(27,30){$v_4$}
\put(32,30){\circle{2}}
\put(32,29){\line(-2,-5){2}}
\put(32,29){\line(2,-5){2}}
\put(31,30){\line(1,0){2}}
\put(32,31){\line(0,-1){2}}
%\emline(20,39)(5,31)
\multiput(20,39)(-.06302521,-.033613445){238}{\line(-1,0){.06302521}}
%\end
%\emline(20,39)(15,31)
\multiput(20,39)(-.033557047,-.053691275){149}{\line(0,-1){.053691275}}
%\end
\put(20,39){\line(1,-2){4}}
\put(20,39){\line(3,-2){12}}
\put(2,0){$T^*$ in Subcase 2.1(i)}
\put(20,40){\circle*{2}}
\put(5,30){\circle*{2}}
%\dottedline(5,29)(5,23)
\multiput(4.93,28.93)(0,-.85714){8}{{\rule{.4pt}{.4pt}}}
%\end
\put(5,22){\circle{2}}
\put(5,21){\line(0,-1){4}}
\put(5,16){\circle*{2}}
\put(5,15){\line(0,-1){4}}
\put(5,10){\circle{2}}
\put(5,21){\line(-1,-1){4}}
\put(7,9){$w$}
\put(7,15){$w'$}
\put(7,21){$w''$}
%\emline(5,21)(1,14)
\multiput(5,21)(-.033613445,-.058823529){119}{\line(0,-1){.058823529}}
%\end
\put(4,22){\line(1,0){2}}
\put(5,23){\line(0,-1){2}}
\put(62,29){\line(-2,-5){2}}
\put(62,29){\line(2,-5){2}}
\put(72,30){\circle{2}}
\put(72,29){\line(-2,-5){2}}
\put(72,29){\line(2,-5){2}}
\put(81,30){\circle{2}}
\put(81,29){\line(-2,-5){2}}
\put(81,29){\line(2,-5){2}}
\put(71,30){\line(1,0){2}}
\put(72,31){\line(0,-1){2}}
\put(80,30){\line(1,0){2}}
\put(81,31){\line(0,-1){2}}
\put(77,42){$u$}
\put(57,30){$v_1$}
\put(67,30){$v_2$}
\put(76,30){$v_3$}
\put(84,30){$v_4$}
\put(89,30){\circle{2}}
\put(89,29){\line(-2,-5){2}}
\put(89,29){\line(2,-5){2}}
\put(88,30){\line(1,0){2}}
\put(89,31){\line(0,-1){2}}
%\emline(77,39)(62,31)
\multiput(77,39)(-.06302521,-.033613445){238}{\line(-1,0){.06302521}}
%\end
%\emline(77,39)(72,31)
\multiput(77,39)(-.033557047,-.053691275){149}{\line(0,-1){.053691275}}
%\end
\put(77,39){\line(1,-2){4}}
\put(77,39){\line(3,-2){12}}
\put(59,0){$T^*$ (the dominating set $D^*_2$)}
\put(62,30){\circle*{2}}
%\dottedline(62,29)(62,23)
\multiput(61.93,28.93)(0,-.85714){8}{{\rule{.4pt}{.4pt}}}
%\end
\put(62,22){\circle{2}}
\put(62,21){\line(0,-1){4}}
\put(62,16){\circle*{2}}
\put(62,15){\line(0,-1){4}}
\put(62,10){\circle{2}}
\put(62,21){\line(-1,-1){4}}
\put(64,9){$w$}
\put(64,15){$w'$}
\put(64,21){$w''$}
%\emline(62,21)(58,14)
\multiput(62,21)(-.033613445,-.058823529){119}{\line(0,-1){.058823529}}
%\end
\put(61,22){\line(1,0){2}}
\put(62,23){\line(0,-1){2}}
\put(77,40){\circle{2}}
\put(43,29){\vector(1,0){5}}
\put(99,29){\vector(1,0){5}}
\put(137,30){\circle{2}}
\put(137,29){\line(-2,-5){2}}
\put(137,29){\line(2,-5){2}}
\put(146,30){\circle{2}}
\put(146,29){\line(-2,-5){2}}
\put(146,29){\line(2,-5){2}}
\put(136,30){\line(1,0){2}}
\put(137,31){\line(0,-1){2}}
\put(145,30){\line(1,0){2}}
\put(146,31){\line(0,-1){2}}
\put(142,42){$u$}
\put(132,30){$v_2$}
\put(141,30){$v_3$}
\put(149,30){$v_4$}
\put(154,30){\circle{2}}
\put(154,29){\line(-2,-5){2}}
\put(154,29){\line(2,-5){2}}
\put(153,30){\line(1,0){2}}
\put(154,31){\line(0,-1){2}}
\put(124,29){\line(-2,-5){2}}
\put(124,29){\line(2,-5){2}}
\put(119,30){$v_1$}
\put(130,37){$w_1$}
\put(124,34){$w_2$}
\put(142,40){\circle{2}}
%\emline(142,39)(137,31)
\multiput(142,39)(-.033557047,-.053691275){149}{\line(0,-1){.053691275}}
%\end
\put(142,39){\line(1,-2){4}}
\put(142,39){\line(3,-2){12}}
\put(130,33){\circle{2}}
\put(135,0){$G_2$}
\put(137,37){\circle*{2}}
\put(137,37){\line(-2,-1){6}}
%\emline(141,39)(137,37)
\multiput(141,39)(-.06666667,-.03333333){60}{\line(-1,0){.06666667}}
%\end
%\dottedline(124,29)(124,23)
\multiput(123.93,28.93)(0,-.85714){8}{{\rule{.4pt}{.4pt}}}
%\end
\put(124,22){\circle{2}}
\put(124,21){\line(-1,-1){4}}
\put(126,21){$w''$}
%\emline(124,21)(120,14)
\multiput(124,21)(-.033613445,-.058823529){119}{\line(0,-1){.058823529}}
%\end
\put(123,22){\line(1,0){2}}
\put(124,23){\line(0,-1){2}}
\put(124,30){\circle*{2}}
\put(129,33){\line(-5,-3){5}}
\end{picture}
  \caption{ \footnotesize Graphs $T^*$ and $G_2$ in Subcase 2.1(i).}\label{case2-1i}
\end{figure}
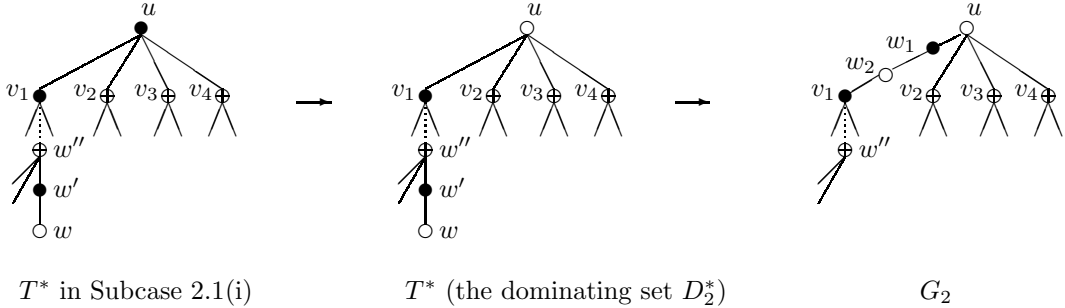

(ii)
If $n_0$ is odd, then  $n_1$ is even and also $\gamma(T_{v_1})=\frac{n_1}{2}$ since $\gamma(T^*)=\frac{n_0+n_1-1}{2}$, $\gamma(T_{v_1})\leq\frac{n_1}{2}$ and $\gamma(T^*-T_{v_1})\leq \frac{n_0-1}{2}$.
By Lemma \ref{half-domination}, we have
$V(T_{v_1})\setminus D^*$ is a minimum dominating set
of $T_{v_1}$,
and so
$$D^*_3= [V(T_{v_1})\setminus D^*]\bigcup [V(T^*-T_{v_1})\cap D^*] \mbox{ (see Fig.\ref{case2-1ii})}$$
is a minimum dominating set
of $T^*$.
Note that $T_{v_i}\not=K_1$. There exists a pendant $3$-path in   $T^*[V(T_{v_i})\cup \{u\}]$ for any $i=1,2,3,4$, and so
 $T^*$ has at least four  pendant $3$-paths.
By Lemma \ref{P3}, at least one of
$T^*[V(T_{v_2})\cup\{u\}]$, $T^*[V(T_{v_3})\cup\{u\}]$ and $T^*[V(T_{v_4})\cup\{u\}]$ contains a pendant $3$-path $w w'w''$ with $d_{T^*}(w)=1$, $d_{T^*}(w')=2$ and $d_{T^*}(w'')\ge2$
such that   $w''$ is not uniquely dominated by $w'$ relative to $D^*$ (or $D^*_3$).
Without loss of generality, we may assume that  $T^*[V(T_{v_2})\cup\{u\}]$ contains such a  $3$-pendant path, say $ww' w''$.
Let $G_3$ be the tree obtained from $T^*$ by
subdividing the edge $uv_1$ twice (the two new vertices, say $w_1$, $w_2$)  and deleting the vertices $w$, $w'$ (see Fig.\ref{case2-1ii}).
Clearly,
$D_3=D^*_3-\{w'\}+\{w_2\}$ is a minimum dominating set of $G_2$.
Thus, $G_3\in \mathbb{G}_{n,\lfloor\frac{n}{2}\rfloor}$.
By Lemmas \ref{subdivision} and \ref{subgraph},
we have $\rho(G_3)< \rho(T^*)$, a contradiction.
\begin{figure}[h]
  \centering
   \footnotesize

\unitlength 0.9mm % = 2.845pt
\linethickness{0.4pt}
\ifx\plotpoint\undefined\newsavebox{\plotpoint}\fi % GNUPLOT compatibility
\begin{picture}(156,42)(0,0)
\put(5,29){\line(-2,-5){2}}
\put(5,29){\line(2,-5){2}}
\put(15,30){\circle{2}}
\put(15,29){\line(-2,-5){2}}
\put(15,29){\line(2,-5){2}}
\put(24,30){\circle{2}}
\put(24,29){\line(-2,-5){2}}
\put(24,29){\line(2,-5){2}}
\put(14,30){\line(1,0){2}}
\put(15,31){\line(0,-1){2}}
\put(23,30){\line(1,0){2}}
\put(24,31){\line(0,-1){2}}
\put(20,42){$u$}
\put(0,30){$v_1$}
\put(10,30){$v_2$}
\put(19,30){$v_3$}
\put(27,30){$v_4$}
\put(32,30){\circle{2}}
\put(32,29){\line(-2,-5){2}}
\put(32,29){\line(2,-5){2}}
\put(31,30){\line(1,0){2}}
\put(32,31){\line(0,-1){2}}
%\emline(20,39)(5,31)
\multiput(20,39)(-.06302521,-.033613445){238}{\line(-1,0){.06302521}}
%\end
%\emline(20,39)(15,31)
\multiput(20,39)(-.033557047,-.053691275){149}{\line(0,-1){.053691275}}
%\end
\put(20,39){\line(1,-2){4}}
\put(20,39){\line(3,-2){12}}
\put(2,0){$T^*$ in Subcase 2.1(ii)}
\put(20,40){\circle*{2}}
\put(5,30){\circle*{2}}
\put(62,29){\line(-2,-5){2}}
\put(62,29){\line(2,-5){2}}
\put(72,30){\circle{2}}
\put(72,29){\line(-2,-5){2}}
\put(72,29){\line(2,-5){2}}
\put(81,30){\circle{2}}
\put(81,29){\line(-2,-5){2}}
\put(81,29){\line(2,-5){2}}
\put(71,30){\line(1,0){2}}
\put(72,31){\line(0,-1){2}}
\put(80,30){\line(1,0){2}}
\put(81,31){\line(0,-1){2}}
\put(77,42){$u$}
\put(57,30){$v_1$}
\put(67,30){$v_2$}
\put(76,30){$v_3$}
\put(84,30){$v_4$}
\put(89,30){\circle{2}}
\put(89,29){\line(-2,-5){2}}
\put(89,29){\line(2,-5){2}}
\put(88,30){\line(1,0){2}}
\put(89,31){\line(0,-1){2}}
%\emline(77,39)(62,31)
\multiput(77,39)(-.06302521,-.033613445){238}{\line(-1,0){.06302521}}
%\end
%\emline(77,39)(72,31)
\multiput(77,39)(-.033557047,-.053691275){149}{\line(0,-1){.053691275}}
%\end
\put(77,39){\line(1,-2){4}}
\put(77,39){\line(3,-2){12}}
\put(59,0){$T^*$ (the dominating set $D^*_3$)}
\put(43,29){\vector(1,0){5}}
\put(99,29){\vector(1,0){5}}
\put(137,30){\circle{2}}
\put(137,29){\line(-2,-5){2}}
\put(137,29){\line(2,-5){2}}
\put(146,30){\circle{2}}
\put(146,29){\line(-2,-5){2}}
\put(146,29){\line(2,-5){2}}
\put(136,30){\line(1,0){2}}
\put(137,31){\line(0,-1){2}}
\put(145,30){\line(1,0){2}}
\put(146,31){\line(0,-1){2}}
\put(142,42){$u$}
\put(132,30){$v_2$}
\put(141,30){$v_3$}
\put(149,30){$v_4$}
\put(154,30){\circle{2}}
\put(154,29){\line(-2,-5){2}}
\put(154,29){\line(2,-5){2}}
\put(153,30){\line(1,0){2}}
\put(154,31){\line(0,-1){2}}
\put(124,29){\line(-2,-5){2}}
\put(124,29){\line(2,-5){2}}
\put(119,30){$v_1$}
\put(130,37){$w_1$}
\put(124,34){$w_2$}
%\emline(142,39)(137,31)
\multiput(142,39)(-.033557047,-.053691275){149}{\line(0,-1){.053691275}}
%\end
\put(142,39){\line(1,-2){4}}
\put(142,39){\line(3,-2){12}}
\put(135,0){$G_3$}
%\dottedline(15,29)(15,23)
\multiput(14.93,28.93)(0,-.85714){8}{{\rule{.4pt}{.4pt}}}
%\end
\put(15,22){\circle{2}}
\put(15,21){\line(0,-1){4}}
\put(15,16){\circle*{2}}
\put(15,15){\line(0,-1){4}}
\put(15,10){\circle{2}}
\put(15,21){\line(-1,-1){4}}
\put(17,9){$w$}
\put(17,15){$w'$}
\put(17,21){$w''$}
%\emline(15,21)(11,14)
\multiput(15,21)(-.033613445,-.058823529){119}{\line(0,-1){.058823529}}
%\end
\put(14,22){\line(1,0){2}}
\put(15,23){\line(0,-1){2}}
%\dottedline(72,29)(72,23)
\multiput(71.93,28.93)(0,-.85714){8}{{\rule{.4pt}{.4pt}}}
%\end
\put(72,22){\circle{2}}
\put(72,21){\line(0,-1){4}}
\put(72,16){\circle*{2}}
\put(72,15){\line(0,-1){4}}
\put(72,10){\circle{2}}
\put(72,21){\line(-1,-1){4}}
\put(74,9){$w$}
\put(74,15){$w'$}
\put(74,21){$w''$}
%\emline(72,21)(68,14)
\multiput(72,21)(-.033613445,-.058823529){119}{\line(0,-1){.058823529}}
%\end
\put(71,22){\line(1,0){2}}
\put(72,23){\line(0,-1){2}}
\put(62,30){\circle{2}}
\put(77,40){\circle*{2}}
\put(124,30){\circle{2}}
\put(142,40){\circle*{2}}
%\emline(129,33)(125,31)
\multiput(129,33)(-.06666667,-.03333333){60}{\line(-1,0){.06666667}}
%\end
\put(130,33){\circle*{2}}
\put(130,33){\line(2,1){6}}
\put(137,36){\circle{2}}
%\emline(142,39)(138,37)
\multiput(142,39)(-.06666667,-.03333333){60}{\line(-1,0){.06666667}}
%\end
%\dottedline(137,29)(137,23)
\multiput(136.93,28.93)(0,-.85714){8}{{\rule{.4pt}{.4pt}}}
%\end
\put(137,22){\circle{2}}
\put(137,21){\line(-1,-1){4}}
\put(139,21){$w''$}
%\emline(137,21)(133,14)
\multiput(137,21)(-.033613445,-.058823529){119}{\line(0,-1){.058823529}}
%\end
\put(136,22){\line(1,0){2}}
\put(137,23){\line(0,-1){2}}
\end{picture}

  \caption{ \footnotesize Graphs $T^*$ and $G_3$ in Subcase 2.1(ii). }\label{case2-1ii}
\end{figure}
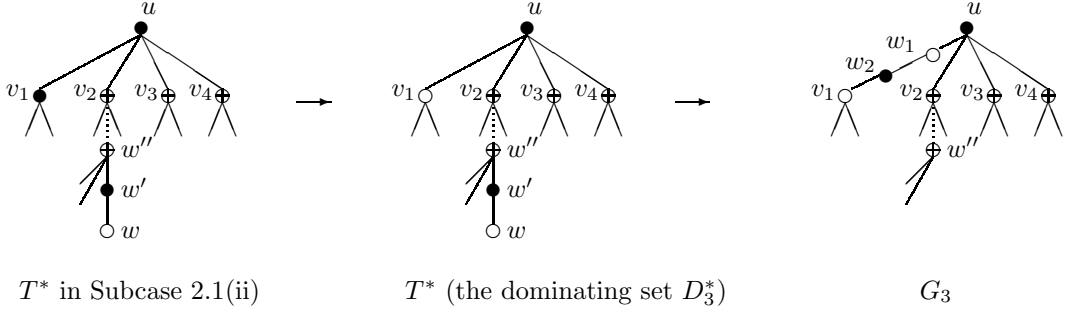

{\flushleft\bf Subcase 2.2 } $u\in D^*$ and $v_1\notin D^*$.

It is clear that $D^*$ in this subcase is similar as
$D^*_3$ in Subcase 2.1(ii).
Now we also construct the tree $G_3$.
Recall that $G_3\in \mathbb{G}_{n,\lfloor\frac{n}{2}\rfloor}$
and
 $\rho(G_3)< \rho(T^*)$, a contradiction.

{\flushleft\bf Subcase 2.3 } $u\notin D^*$.

As discussed in  Subcase 2.1(ii), by Lemma \ref{P3}
we may assume that
 $T^*[V(T_{v_2})\cup \{u\}]$ contains a pendant $3$-path, say $w w' w''$  with $d_{T^*}(w)=1$, $d_{T^*}(w')=2$ and $d_{T^*}(w'')\ge2$ such that   $w''$ is not uniquely dominated by $w'$ relative to $D^*$.
Let $G_4$ be the tree obtained from $T^*$ by
subdividing the edge $uv_1$ twice (say $w_1$, $w_2$ the two new vertices), and deleting the vertices $w$, $w'$ (see Fig.\ref{case2-3}).
It is easy to verify that
$D_4=D^*-\{w'\}+\{w_1\}$ is a minimum dominating set of $G_4$ whether $v_1\in D^*$ or not (see Fig.\ref{case2-3}).
Thus, $G_4\in \mathbb{G}_{n,\lfloor\frac{n}{2}\rfloor}$.
By Lemmas \ref{subdivision} and \ref{subgraph},
we have $\rho(G_4)< \rho(T^*)$, a contradiction.

It completes the proof.
\end{proof}

\begin{figure}[h]
  \centering
   \footnotesize
\unitlength 0.88mm % = 2.845pt
\linethickness{0.4pt}
\ifx\plotpoint\undefined\newsavebox{\plotpoint}\fi % GNUPLOT compatibility
\begin{picture}(179,42)(0,0)
\put(5,29){\line(-2,-5){2}}
\put(5,29){\line(2,-5){2}}
\put(15,30){\circle{2}}
\put(15,29){\line(-2,-5){2}}
\put(15,29){\line(2,-5){2}}
\put(24,30){\circle{2}}
\put(24,29){\line(-2,-5){2}}
\put(24,29){\line(2,-5){2}}
\put(14,30){\line(1,0){2}}
\put(15,31){\line(0,-1){2}}
\put(23,30){\line(1,0){2}}
\put(24,31){\line(0,-1){2}}
\put(20,42){$u$}
\put(0,30){$v_1$}
\put(10,30){$v_2$}
\put(19,30){$v_3$}
\put(27,30){$v_4$}
\put(36,30){\vector(1,0){5}}
\put(32,30){\circle{2}}
\put(32,29){\line(-2,-5){2}}
\put(32,29){\line(2,-5){2}}
\put(31,30){\line(1,0){2}}
\put(32,31){\line(0,-1){2}}
%\dottedline(15,28)(15,22)
\multiput(14.93,27.93)(0,-.85714){8}{{\rule{.4pt}{.4pt}}}
%\end
\put(15,21){\circle{2}}
\put(15,20){\line(0,-1){4}}
\put(15,15){\circle*{2}}
\put(15,14){\line(0,-1){4}}
\put(15,9){\circle{2}}
\put(15,20){\line(-1,-1){4}}
\put(17,8){$w$}
\put(17,14){$w'$}
\put(17,20){$w''$}
\put(20,40){\circle{2}}
%\emline(20,39)(5,31)
\multiput(20,39)(-.06302521,-.033613445){238}{\line(-1,0){.06302521}}
%\end
%\emline(20,39)(15,31)
\multiput(20,39)(-.033557047,-.053691275){149}{\line(0,-1){.053691275}}
%\end
\put(20,39){\line(1,-2){4}}
\put(20,39){\line(3,-2){12}}
\put(6,0){$T^*$ ($v_1\notin D^*$)}
%\emline(15,20)(11,13)
\multiput(15,20)(-.033613445,-.058823529){119}{\line(0,-1){.058823529}}
%\end
\put(63,30){\circle{2}}
\put(63,29){\line(-2,-5){2}}
\put(63,29){\line(2,-5){2}}
\put(72,30){\circle{2}}
\put(72,29){\line(-2,-5){2}}
\put(72,29){\line(2,-5){2}}
\put(62,30){\line(1,0){2}}
\put(63,31){\line(0,-1){2}}
\put(71,30){\line(1,0){2}}
\put(72,31){\line(0,-1){2}}
\put(68,42){$u$}
\put(58,30){$v_2$}
\put(67,30){$v_3$}
\put(75,30){$v_4$}
\put(80,30){\circle{2}}
\put(80,29){\line(-2,-5){2}}
\put(80,29){\line(2,-5){2}}
\put(79,30){\line(1,0){2}}
\put(80,31){\line(0,-1){2}}
%\dottedline(63,28)(63,22)
\multiput(62.93,27.93)(0,-.85714){8}{{\rule{.4pt}{.4pt}}}
%\end
\put(63,21){\circle{2}}
\put(63,20){\line(-1,-1){4}}
\put(65,20){$w''$}
\put(50,29){\line(-2,-5){2}}
\put(50,29){\line(2,-5){2}}
\put(45,30){$v_1$}
\put(56,37){$w_1$}
\put(50,34){$w_2$}
\put(68,40){\circle{2}}
%\emline(68,39)(63,31)
\multiput(68,39)(-.033557047,-.053691275){149}{\line(0,-1){.053691275}}
%\end
\put(68,39){\line(1,-2){4}}
\put(68,39){\line(3,-2){12}}
\put(56,33){\circle{2}}
\put(61,0){$G_4$}
%\emline(63,20)(59,13)
\multiput(63,20)(-.033613445,-.058823529){119}{\line(0,-1){.058823529}}
%\end
\put(63,37){\circle*{2}}
\put(63,37){\line(-2,-1){6}}
%\emline(67,39)(63,37)
\multiput(67,39)(-.06666667,-.03333333){60}{\line(-1,0){.06666667}}
%\end
\put(14,21){\line(1,0){2}}
\put(15,22){\line(0,-1){2}}
\put(62,21){\line(1,0){2}}
\put(63,22){\line(0,-1){2}}
\put(5,30){\circle{2}}
\put(50,30){\circle{2}}
%\emline(55,33)(51,31)
\multiput(55,33)(-.06666667,-.03333333){60}{\line(-1,0){.06666667}}
%\end
\put(102,29){\line(-2,-5){2}}
\put(102,29){\line(2,-5){2}}
\put(112,30){\circle{2}}
\put(112,29){\line(-2,-5){2}}
\put(112,29){\line(2,-5){2}}
\put(121,30){\circle{2}}
\put(121,29){\line(-2,-5){2}}
\put(121,29){\line(2,-5){2}}
\put(111,30){\line(1,0){2}}
\put(112,31){\line(0,-1){2}}
\put(120,30){\line(1,0){2}}
\put(121,31){\line(0,-1){2}}
\put(117,42){$u$}
\put(97,30){$v_1$}
\put(107,30){$v_2$}
\put(116,30){$v_3$}
\put(124,30){$v_4$}
\put(133,30){\vector(1,0){5}}
\put(129,30){\circle{2}}
\put(129,29){\line(-2,-5){2}}
\put(129,29){\line(2,-5){2}}
\put(128,30){\line(1,0){2}}
\put(129,31){\line(0,-1){2}}
%\dottedline(112,28)(112,22)
\multiput(111.93,27.93)(0,-.85714){8}{{\rule{.4pt}{.4pt}}}
%\end
\put(112,21){\circle{2}}
\put(112,20){\line(0,-1){4}}
\put(112,15){\circle*{2}}
\put(112,14){\line(0,-1){4}}
\put(112,9){\circle{2}}
\put(112,20){\line(-1,-1){4}}
\put(114,8){$w$}
\put(114,14){$w'$}
\put(114,20){$w''$}
\put(117,40){\circle{2}}
\put(102,30){\circle*{2}}
%\emline(117,39)(102,31)
\multiput(117,39)(-.06302521,-.033613445){238}{\line(-1,0){.06302521}}
%\end
%\emline(117,39)(112,31)
\multiput(117,39)(-.033557047,-.053691275){149}{\line(0,-1){.053691275}}
%\end
\put(117,39){\line(1,-2){4}}
\put(117,39){\line(3,-2){12}}
\put(105,0){$T^*$ ($v_1\in D^*$)}
%\emline(112,20)(108,13)
\multiput(112,20)(-.033613445,-.058823529){119}{\line(0,-1){.058823529}}
%\end
\put(160,30){\circle{2}}
\put(160,29){\line(-2,-5){2}}
\put(160,29){\line(2,-5){2}}
\put(169,30){\circle{2}}
\put(169,29){\line(-2,-5){2}}
\put(169,29){\line(2,-5){2}}
\put(159,30){\line(1,0){2}}
\put(160,31){\line(0,-1){2}}
\put(168,30){\line(1,0){2}}
\put(169,31){\line(0,-1){2}}
\put(165,42){$u$}
\put(155,30){$v_2$}
\put(164,30){$v_3$}
\put(172,30){$v_4$}
\put(177,30){\circle{2}}
\put(177,29){\line(-2,-5){2}}
\put(177,29){\line(2,-5){2}}
\put(176,30){\line(1,0){2}}
\put(177,31){\line(0,-1){2}}
%\dottedline(160,28)(160,22)
\multiput(159.93,27.93)(0,-.85714){8}{{\rule{.4pt}{.4pt}}}
%\end
\put(160,21){\circle{2}}
\put(160,20){\line(-1,-1){4}}
\put(162,20){$w''$}
\put(147,29){\line(-2,-5){2}}
\put(147,29){\line(2,-5){2}}
\put(142,30){$v_1$}
\put(153,37){$w_1$}
\put(147,34){$w_2$}
\put(147,30){\circle*{2}}
\put(165,40){\circle{2}}
%\emline(165,39)(160,31)
\multiput(165,39)(-.033557047,-.053691275){149}{\line(0,-1){.053691275}}
%\end
\put(165,39){\line(1,-2){4}}
\put(165,39){\line(3,-2){12}}
\put(153,33){\circle{2}}
\put(152,33){\line(-5,-3){5}}
\put(158,0){$G_4$}
%\emline(160,20)(156,13)
\multiput(160,20)(-.033613445,-.058823529){119}{\line(0,-1){.058823529}}
%\end
\put(160,37){\circle*{2}}
\put(160,37){\line(-2,-1){6}}
%\emline(164,39)(160,37)
\multiput(164,39)(-.06666667,-.03333333){60}{\line(-1,0){.06666667}}
%\end
\put(111,21){\line(1,0){2}}
\put(112,22){\line(0,-1){2}}
\put(159,21){\line(1,0){2}}
\put(160,22){\line(0,-1){2}}
\end{picture}

  \caption{ \footnotesize Graphs $T^*$ and $G_4$ in Subcase 2.3. }\label{case2-3}
\end{figure}
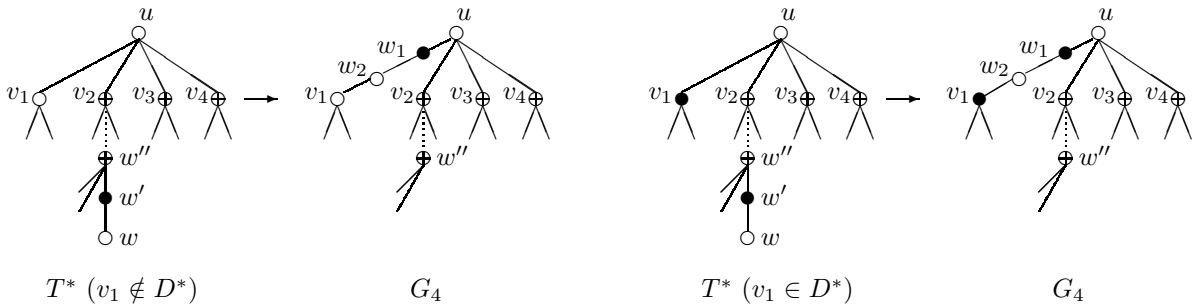

\begin{lem}\label{diameter-bound}
For odd $n \ge 13$,
let $T^*$ be the minimizer graph in $\mathbb{G}_{n,\lfloor\frac{n}{2}\rfloor}$.
Then  the diameter $d(T^*)$ of $T^*$ satisfies $d(T^*)\ge7$.
\end{lem}
\begin{proof}
Let $d=d(T^*)$ and  $P_{d+1}=v_1v_2\cdots v_{d+1}$ be a diametrical path of $T^*$.
By Lemmas \ref{leave-number} and \ref{max-degree}, we
have
$$n\le\left\{\begin{array}{ll}
%n(T_3)=7,&\mbox{if $d=4$,}\\
n(T_1)=10,&\mbox{if $d=5$,}\\
n(T_2)=16,&\mbox{if $d=6$,}
\end{array}\right.$$
where $T_1$ and $T_2$ are described in Fig.\ref{TT}.
Thus,  we have $d\ge7$ for $n\ge17$.
On the other hand, it is easy to find the tree set
\begin{align*}
\mathbf{T}_n&=\{T: n(T)=n,\ d(T)\le 6,\ \gamma(T)=\frac{n-1}{2},\ \Delta(T)= 3,\  l_T(u)\le1\ \mbox{for}\  u\in V(T) \}\\
&=
\left\{\begin{array}{ll}
\emptyset, & \mbox{if $ n=15$},\\
\{T_3\}, & \mbox{if $ n=13$},
\end{array}\right.\end{align*}
where $T_3$ is shown in Fig.\ref{TT}.
A simple calculation shows
$$\rho(T_3)=2.2882>2.1358=\rho(T_{\lceil\frac{n-3}{4}\rceil, \lfloor\frac{n-3}{4}\rfloor}^{\frac{n+3}{2}})$$ for $n=13$.
Thus, any graph in $\mathbf{T}_n$ is not a minimizer graph for $13\le n\le 15$.
Since every graph in $\mathbf{T}_n$ has the diameter at most $6$,
 we have $d\ge7$ for $13\le n\le 15$.

It completes the proof.
\end{proof}

\begin{figure}[h]
  \centering
  \footnotesize
% This is a LaTeX picture output by TeXCAD.
% File name: [Clipboard].
% Version of TeXCAD: 4.3
% Reference / build: 30-Jun-2012 (rev. 105)
% For new versions, check: http://texcad.sf.net/
% Options on the following lines.
%\grade{\on}
%\emlines{\off}
%\epic{\off}
%\beziermacro{\on}
%\reduce{\on}
%\snapping{\on}
%\pvinsert{% Your \input, \def, etc. here}
%\quality{8.000}
%\graddiff{0.005}
%\snapasp{1}
%\zoom{5.6569}
\unitlength 0.7mm % = 2.845pt
\linethickness{0.4pt}
\ifx\plotpoint\undefined\newsavebox{\plotpoint}\fi % GNUPLOT compatibility
\begin{picture}(193,28)(0,0)
\put(1,27){\line(1,0){9}}
\put(1,27){\circle*{2}}
\put(10,27){\circle*{2}}
\put(10,27){\line(1,0){9}}
\put(19,27){\circle*{2}}
\put(19,27){\line(1,0){9}}
\put(28,27){\circle*{2}}
\put(28,27){\line(1,0){9}}
\put(37,27){\circle*{2}}
\put(37,27){\line(1,0){9}}
\put(46,27){\circle*{2}}
\put(19,27){\line(0,-1){6}}
\put(19,21){\circle*{2}}
\put(19,21){\line(0,-1){6}}
\put(19,15){\circle*{2}}
\put(28,27){\line(0,-1){6}}
\put(28,21){\circle*{2}}
\put(28,21){\line(0,-1){6}}
\put(28,15){\circle*{2}}
\put(22,0){$T_1$}
\put(65,27){\line(1,0){9}}
\put(65,27){\circle*{2}}
\put(74,27){\circle*{2}}
\put(74,27){\line(1,0){9}}
\put(83,27){\circle*{2}}
\put(83,27){\line(1,0){9}}
\put(92,27){\circle*{2}}
\put(92,27){\line(1,0){9}}
\put(101,27){\circle*{2}}
\put(101,27){\line(1,0){9}}
\put(110,27){\circle*{2}}
\put(110,27){\line(1,0){9}}
\put(119,27){\circle*{2}}
\put(83,27){\line(0,-1){6}}
\put(83,21){\circle*{2}}
\put(83,21){\line(0,-1){6}}
\put(83,15){\circle*{2}}
\put(92,27){\line(0,-1){6}}
\put(92,21){\circle*{2}}
\put(92,21){\line(0,-1){6}}
\put(92,15){\circle*{2}}
\put(101,27){\line(0,-1){6}}
\put(101,21){\circle*{2}}
\put(101,21){\line(0,-1){6}}
\put(101,15){\circle*{2}}
\put(92,15){\line(0,-1){6}}
\put(92,9){\circle*{2}}
\put(92,21){\line(-2,-3){4}}
\put(88,15){\circle*{2}}
\put(88,15){\line(0,-1){6}}
\put(88,9){\circle*{2}}
\put(91,0){$T_2$}
\put(138,27){\line(1,0){9}}
\put(138,27){\circle*{2}}
\put(147,27){\circle*{2}}
\put(147,27){\line(1,0){9}}
\put(156,27){\circle*{2}}
\put(156,27){\line(1,0){9}}
\put(165,27){\circle*{2}}
\put(165,27){\line(1,0){9}}
\put(174,27){\circle*{2}}
\put(174,27){\line(1,0){9}}
\put(183,27){\circle*{2}}
\put(183,27){\line(1,0){9}}
\put(192,27){\circle*{2}}
\put(156,27){\line(0,-1){6}}
\put(156,21){\circle*{2}}
\put(165,27){\line(0,-1){6}}
\put(165,21){\circle*{2}}
\put(165,21){\line(0,-1){6}}
\put(165,15){\circle*{2}}
\put(174,27){\line(0,-1){6}}
\put(174,21){\circle*{2}}
\put(165,15){\line(0,-1){6}}
\put(165,9){\circle*{2}}
\put(165,21){\line(-2,-3){4}}
\put(161,15){\circle*{2}}
\put(164,0){$T_3$}
\end{picture}

  \caption{\footnotesize Graphs $T_1,T_2, T_3$.}\label{TT}
\end{figure}
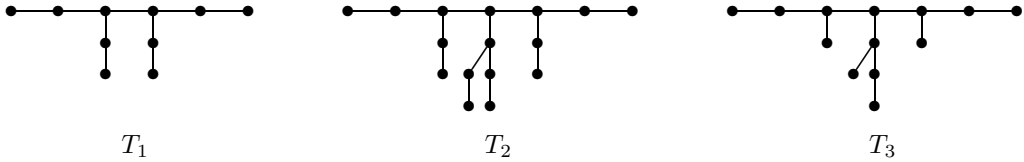

A tree  is called  a \emph{caterpillar} if it becomes
a path after deleting all leaves.

\begin{lem}\label{caterpillar}
For odd $n \ge 13$,
let $T^*$ be the minimizer graph in $\mathbb{G}_{n,\lfloor\frac{n}{2}\rfloor}$.
Then  $T^*$ is  a caterpillar.
\end{lem}

\begin{proof}
Let $P_{d+1}=v_1v_2\cdots v_{d+1}$ be a diametrical path of $T^*$.
Suppose to the contrary that $T^*$ is not  a caterpillar,
that is, there exists some vertex $v_i$ of the path $P_{d+1}$ such that $T_{u}\not=K_1$, where $T_{u}$ is the component of $T^*-v_i$ containing the unique neighbor $u$ of $v_i$ not in $P_{d+1}$.
Firstly, we have  the following claim.

\begin{clm}\label{claim-1}
$T^*-T_{u}$ is not isomorphic to a path.
\end{clm}
\begin{proof}
%[\bf{Proof of Claim}]
Suppose to the contrary that $T^*-T_{u}$ is a path, i.e.,
$T^*-T_{u}=P_{d+1}$.
Then
$$\gamma(T^*)\le \gamma(P_{d+1}) + \gamma(T_{u})
\le \frac{d}{3}+1+\frac{n-(d+1)}{2}=\frac{n}{2}-\frac{d-3}{6}<\frac{n-1}{2}=\gamma(T^*),$$
for $d\ge7$ (due to Lemma \ref{diameter-bound}), a contradiction.
Thus Claim \ref{claim-1} holds.
\end{proof}

Recall that $\Delta(T^*)=3$. By Claim \ref{claim-1}, there must exist some vertex $v_j\not=v_i$ in  $P_{d+1}$ such that $d_{T^*}(v_j)=3$.
Without loss of generality, we may assume that $j<i$. Thus,  $v_{i-1}v_i$  is an edge on an internal path of $T^*$.
Since $T_u\not=K_1$, there exists a pendant $3$-path $ww'w''$ with $d_{T^*}(w)=1$,
$d_{T^*}(w')=2$  and $d_{T^*}(w'')\ge 2$ in the induced subgraph $T^*[V(T_u)\cup \{v_i\}]$.
By Lemma \ref{P3-exist}, there exists
a minimum dominating set $D$  of $T^*$   such that  $w\notin D$, $w'\in D$ and
  $w''$ is not uniquely dominated by $w'$ relative to  $D$.
Let $T'$  be the tree obtained from $T^*$ by
subdividing the edge $v_{i-1}v_i$ twice (say $w_1$, $w_2$ the two new vertices), and deleting the vertices $w$, $w'$.
  Next we will show $\gamma(T')=\gamma(T^*)$ from the following two cases.
{\flushleft\bf Case 1. }
At least one of $v_{i-1}$ and $v_i$ does not belong to $D$.

Without loss of generality  assume $v_{i-1}\notin D$.
Let $D'=D-\{w'\}+\{w_1\}$.
It is easy to verify that $D'$ is a  minimum dominating set of $T'$.
Thus, $\gamma(T')=\gamma(T^*)$.
{\flushleft\bf Case 2. }
Both $v_{i-1}$ and $v_i$  belong to $D$.

Let $T_{v_{i-1}}$ and $T_{v_{i}}$ be the components of $T^*-v_{i-1}v_i$ containing  $v_{i-1}$ and $v_{i}$, respectively.
In this case,
 $V(T_{v_{i-1}})\cap D$ and $ V(T_{v_{i}})\cap D$ is a  minimum dominating set of $T_{v_{i-1}}$ and $T_{v_{i}}$, respectively.
 By Lemma \ref{domination-number-upbound},
 we have $\gamma(T_{v_{i-1}})\le \lfloor\frac{n(T_{v_{i-1}})}{2}\rfloor$
and
$\gamma(T_{v_{i}})\le\lfloor\frac{n(T_{v_{i}})}{2}\rfloor$.
Recall that $\gamma(T^*)=\frac{n-1}{2}$,
we have $\gamma(T_{v_{i-1}})=\lfloor\frac{n(T_{v_{i-1}})}{2}\rfloor$
and
$\gamma(T_{v_{i}})=\lfloor\frac{n(T_{v_{i}})}{2}\rfloor$.
If $n(T_{v_{i-1}})$  is even, then
 $\gamma(T_{v_{i-1}})=\frac{n(T_{v_{i-1}})}{2}$.
 By Lemma \ref{half-domination}, we have $V(T_{v_{i-1}})\setminus D$ is also a  minimum dominating set of $T_{v_{i-1}}$, and so
$$D_1=[V(T_{v_{i-1}})\setminus D]\bigcup [ V(T_{v_{i}})\cap D]$$
is  also a  minimum dominating set of $T^*$. Thus,
$v_{i-1}\notin D_1$ and  $v_{i}\in D_1$.
One can see that $w\notin D_1$, $w'\in D_1$ and $w''$ is not uniquely dominated by $w'$ relative to  $D_1$.
Let $D'_1=D_1-\{w'\}+\{w_1\}$.
It is easy to verify that $D'_1$ is a  minimum dominating set of $T'$.
Thus, $\gamma(T')=\gamma(T^*)$.
If $n(T_{v_{i-1}})$  is odd, then
$n(T_{v_{i}})=n-n(T_{v_{i-1}})$ is even, and so
$\gamma(T_{v_{i}})=\frac{n(T_{v_{i}})}{2}$.
Thus, $V(T_{v_{i}})\setminus D$
is also  a dominating set of
$T_{v_{i}}$ from Lemma \ref{half-domination}.
Let $$D_2=[V(T_{v_{i-1}})\cap D]\bigcup [V(T_{v_{i}})\setminus D].$$
Clearly, $D_2$ is a  minimum dominating set of
$T^*$ and $v_{i-1}\in D_2$, $v_{i}\notin D_2$.
It is clear that
 $w\in D_2$ and  $w'\notin D_2$.
 Let $D'_2=D_2-\{w\}+\{w_2\}$.
 Then $D'_2$ is a  minimum dominating set of $T'$, and so
 $\gamma(T')=\gamma(T^*)$.

Now we have known  $\gamma(T')=\gamma(T^*)$, which implies  $T'\in \mathbb{G}_{n,\lfloor\frac{n}{2}\rfloor}$.
By Lemmas \ref{subdivision} and \ref{subgraph},
we have $\rho(T')< \rho(T^*)$, a contradiction.
It completes the proof.
\end{proof}

%Let $H_{i,\frac{n-1}{2}-i}$ ($i\in[ \lceil\frac{n-1}{4}\rceil,\frac{n-1}{2}]$) be a tree on $n$ vertices
%obtained from
%the path $P_{\frac{n+1}{2}}$ by attaching a pendant vertex to
%the  first $i$ vertices and the last $\frac{n-1}{2}-i$
%vertices of $P_{\frac{n+1}{2}}$, respectively.
%\textcolor[rgb]{1.00,0.00,0.00}{$v_k$ for $k\in [1,i]\cup[i+2, \frac{n+1}{2}]$ }
%(see Fig. \ref{H}).
\begin{figure}[h]
  \centering
    \footnotesize
  % This is a LaTeX picture output by TeXCAD.
% File name: [Clipboard].
% Version of TeXCAD: 4.3
% Reference / build: 30-Jun-2012 (rev. 105)
% For new versions, check: http://texcad.sf.net/
% Options on the following lines.
%\grade{\on}
%\emlines{\off}
%\epic{\off}
%\beziermacro{\on}
%\reduce{\on}
%\snapping{\on}
%\pvinsert{% Your \input, \def, etc. here}
%\quality{8.000}
%\graddiff{0.005}
%\snapasp{1}
%\zoom{8.0000}
\unitlength 1mm % = 2.845pt
\linethickness{0.4pt}
\ifx\plotpoint\undefined\newsavebox{\plotpoint}\fi % GNUPLOT compatibility
\begin{picture}(63,18)(0,0)
\put(1,15){\circle*{2}}
\put(1,15){\line(0,-1){8}}
\put(1,7){\circle*{2}}
\put(1,15){\line(1,0){8}}
\put(9,15){\circle*{2}}
\put(9,15){\line(0,-1){8}}
\put(9,7){\circle*{2}}
\put(25,15){\circle*{2}}
\put(25,15){\line(0,-1){8}}
\put(25,7){\circle*{2}}
%\dashline{1}(9,15)(25,15)
\put(8.93,14.93){\line(1,0){.9412}}
\put(10.812,14.93){\line(1,0){.9412}}
\put(12.694,14.93){\line(1,0){.9412}}
\put(14.577,14.93){\line(1,0){.9412}}
\put(16.459,14.93){\line(1,0){.9412}}
\put(18.341,14.93){\line(1,0){.9412}}
\put(20.224,14.93){\line(1,0){.9412}}
\put(22.106,14.93){\line(1,0){.9412}}
\put(23.989,14.93){\line(1,0){.9412}}
%\end
\put(0,18){$v_{1}$}
\put(8,18){$v_{2}$}
\put(24,18){$v_{i}$}
\put(31,18){$v_{i+1}$}
%\put(25,0){$H_{i,\frac{n-1}{2}-i}$}
\put(25,15){\line(1,0){8}}
\put(33,15){\circle*{2}}
\put(40,18){$v_{i+2}$}
\put(55,18){$v_{\frac{n+1}{2}}$}
\put(33,15){\line(1,0){8}}
\put(41,15){\circle*{2}}
\put(41,15){\line(0,-1){8}}
\put(41,7){\circle*{2}}
\put(57,15){\circle*{2}}
\put(57,15){\line(0,-1){8}}
\put(57,7){\circle*{2}}
%\dashline{1}(41,15)(57,15)
\put(40.93,14.93){\line(1,0){.9412}}
\put(42.812,14.93){\line(1,0){.9412}}
\put(44.694,14.93){\line(1,0){.9412}}
\put(46.577,14.93){\line(1,0){.9412}}
\put(48.459,14.93){\line(1,0){.9412}}
\put(50.341,14.93){\line(1,0){.9412}}
\put(52.224,14.93){\line(1,0){.9412}}
\put(54.106,14.93){\line(1,0){.9412}}
\put(55.989,14.93){\line(1,0){.9412}}
%\end
\end{picture}

  \caption{  \footnotesize Graph $T^{\frac{n+1}{2}}_{i,\frac{n-1}{2}-i}$.}\label{H}
\end{figure}
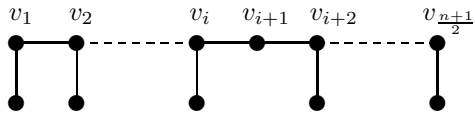

\begin{lem}\label{diameter}
For odd $n \ge 13$,
let $T^*$ be the minimizer graph in $\mathbb{G}_{n,\lfloor\frac{n}{2}\rfloor}$.
Then  $d(T^*)=\frac{n+5}{2}$.
\end{lem}

\begin{proof}
Let $P_{d+1}=v_1v_2\cdots v_{d+1}$ be a  diametrical path of $T^*$.
By Lemmas \ref{caterpillar} and \ref{max-degree},
$T^*$ is a caterpillar with maximum degree $3$.
Since $v_1$ and $v_{d+1}$ are leaves in $T^*$,
by Lemma \ref{leave-number}, $v_2$ and $v_{d}$ can not attach any leaf.
Let $L(T^*)$ be the set of  leaves of $T^*$. Then $|L(T^*)\backslash\{v_1,v_{d+1}\}|=n-d-1$ and their neighbors are the subset of $\{v_3,v_4,\ldots,v_{d-1}\}$.
On the one hand, $T^*$ has $n-d+1$ leaves,
from Lemma \ref{dominating-set-support},
we have
\begin{equation}\label{gamma1}
\gamma(T^*)\ge n-d+1.
\end{equation}
On the other hand,
except $n-d+1$ dominating vertices attaching   leaves, the path $P=v_3v_4\cdots v_{d-1}$ also contains at most $\lfloor\frac{(d-3)-(n-d-1)-2}{3}\rfloor+1$ dominating vertices.
Thus,
%$T^*$ has  at most $n-(d+1)+2+\frac{(d-3)-(n-(d+1))}{3}$ dominating vertices, i.e.,
\begin{equation}\label{gamma2}
\gamma(T^*)\le n-d+2+\lfloor\frac{2d-n-4}{3}\rfloor.
\end{equation}
Note that $\gamma(T^*)=\frac{n-1}{2}$,
combining (\ref{gamma1}) and (\ref{gamma2}) respectively, we have
$\frac{n+3}{2}\le d\le \frac{n+5}{2}$.

Suppose to the contrary  that  $d=\frac{n+3}{2}$. Then
$T^*$ is obtained from $P_{\frac{n+5}{2}}$ by
adding $n-\frac{n+5}{2}=\frac{n-5}{2}$ leaves to
$\frac{n-5}{2}$ vertices in $\{v_3,v_4,\ldots, v_{\frac{n+5}{2}-2}\}$, that is,
 $T^*\cong T^{\frac{n+1}{2}}_{i,\frac{n-1}{2}-i}$ (see Fig.\ref{H}) for some $i\in[\lceil\frac{n-1}{4}\rceil, \frac{n-1}{2}]$.
Note that
$T^{\frac{n+3}{2}}_{i-1,\frac{n-3}{2}-i+1}$ is obtained from $T^{\frac{n+1}{2}}_{i,\frac{n-1}{2}-i}$ by subdividing the edge $v_{i-1}v_i$ and deleting the leaf attached  to $v_i$.
By Lemmas \ref{subdivision} and \ref{subgraph},
we have $\rho(T^{\frac{n+3}{2}}_{i-1,\frac{n-3}{2}-i+1})< \rho(T^{\frac{n+1}{2}}_{i,\frac{n-1}{2}-i})=\rho(T^*)$, a contradiction.
Thus  $d(T^*)=d=\frac{n+5}{2}$.
\end{proof}

\begin{proof}[\bf{Proof of Theorem \ref{extremal-graph}}]
Assume that  $3\le n \le 7$,
since $\gamma(P_n)=\lceil\frac{n}{3}\rceil=\frac{n-1}{2}$,
the path $P_n$ is the desired graph.
Note that  $P_n\cong T^{\frac{n+3}{2}}_{\lceil\frac{n-3}{4}\rceil, \lfloor\frac{n-3}{4}\rfloor }$.
The result holds for $3\le n \le 7$.

If $9\le n\le11$.
By Lemmas \ref{min-tree} and \ref{leave-number},
it is easy to find the tree sets
$$\{T: n(T)=9, \gamma(T)=4, l_T(u)\le 1 \
\mbox{for} \ u\in V(T)\}=\{H_1,\ldots,H_7\} \ \mbox{(see Fig.\ref{H9-11})}$$
and
$$\{T: n(T)=11, \gamma(T)=5, l_T(u)\le 1 \
\mbox{for} \ u\in V(T)\}=\{H_8,\ldots,H_{25}\} \ \mbox{(see Fig.\ref{H9-11})}.$$
From Fig.\ref{H9-11},
 $H_7$ and $H_{25}$ is the minimizer graph for $n=9$ and $n=11$, respectively.
It is clear that  $H_7\cong T^{\frac{n+3}{2}}_{\lceil\frac{n-3}{4}\rceil, \lfloor\frac{n-3}{4}\rfloor }$ for $n=9$ and
$H_{25} \cong T^{\frac{n+3}{2}}_{\lceil\frac{n-3}{4}\rceil, \lfloor\frac{n-3}{4}\rfloor }$ for $n=11$.
The result holds for $9\le n \le 11$.

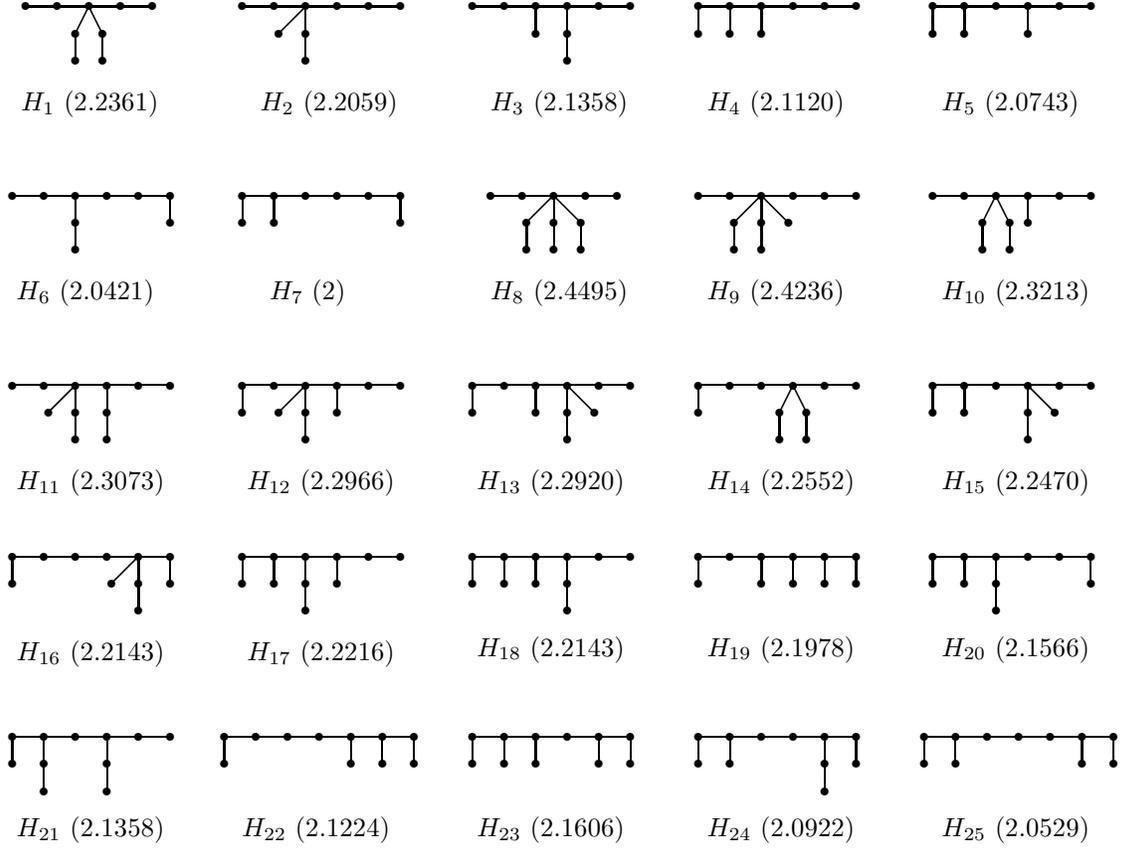
\begin{figure}
  \centering
   \footnotesize

\unitlength 0.6mm % = 2.845pt
\linethickness{0.4pt}
\ifx\plotpoint\undefined\newsavebox{\plotpoint}\fi % GNUPLOT compatibility
\begin{picture}(246,185)(0,0)
\put(110,100){\circle*{2}}
\put(110,100){\line(1,0){7}}
\put(117,100){\circle*{2}}
\put(117,100){\line(1,0){7}}
\put(124,100){\circle*{2}}
\put(124,100){\line(1,0){7}}
\put(131,100){\circle*{2}}
\put(131,100){\line(1,0){7}}
\put(138,100){\circle*{2}}
\put(104,77){$H_{13}$ ($2.2920$)}
\put(103,100){\circle*{2}}
\put(103,100){\line(1,0){7}}
\put(117,100){\line(0,-1){6}}
\put(117,94){\circle*{2}}
\put(103,100){\line(0,-1){6}}
\put(103,94){\circle*{2}}
\put(124,100){\line(0,-1){6}}
\put(124,94){\circle*{2}}
\put(124,94){\line(0,-1){6}}
\put(124,88){\circle*{2}}
\put(124,100){\line(1,-1){6}}
\put(130,94){\circle*{2}}
\put(160,100){\circle*{2}}
\put(160,100){\line(1,0){7}}
\put(167,100){\circle*{2}}
\put(167,100){\line(1,0){7}}
\put(174,100){\circle*{2}}
\put(174,100){\line(1,0){7}}
\put(181,100){\circle*{2}}
\put(181,100){\line(1,0){7}}
\put(188,100){\circle*{2}}
\put(155,77){$H_{14}$ ($2.2552$)}
\put(174,100){\line(-1,-2){3}}
\put(174,100){\line(1,-2){3}}
\put(171,94){\circle*{2}}
\put(171,94){\line(0,-1){6}}
\put(171,88){\circle*{2}}
\put(177,94){\circle*{2}}
\put(177,94){\line(0,-1){6}}
\put(177,88){\circle*{2}}
\put(153,100){\circle*{2}}
\put(153,100){\line(1,0){7}}
\put(153,100){\line(0,-1){6}}
\put(153,94){\circle*{2}}
\put(212,100){\circle*{2}}
\put(212,100){\line(1,0){7}}
\put(219,100){\circle*{2}}
\put(219,100){\line(1,0){7}}
\put(226,100){\circle*{2}}
\put(226,100){\line(1,0){7}}
\put(233,100){\circle*{2}}
\put(233,100){\line(1,0){7}}
\put(240,100){\circle*{2}}
\put(207,77){$H_{15}$ ($2.2470$)}
\put(205,100){\circle*{2}}
\put(205,100){\line(1,0){7}}
\put(205,100){\line(0,-1){6}}
\put(205,94){\circle*{2}}
\put(226,100){\line(0,-1){6}}
\put(226,94){\circle*{2}}
\put(226,94){\line(0,-1){6}}
\put(226,88){\circle*{2}}
\put(226,100){\line(1,-1){6}}
\put(232,94){\circle*{2}}
\put(212,100){\line(0,-1){6}}
\put(212,94){\circle*{2}}
\put(52,100){\circle*{2}}
\put(52,100){\line(1,0){7}}
\put(59,100){\circle*{2}}
\put(59,100){\line(1,0){7}}
\put(66,100){\circle*{2}}
\put(66,100){\line(1,0){7}}
\put(73,100){\circle*{2}}
\put(73,100){\line(1,0){7}}
\put(80,100){\circle*{2}}
\put(53,77){$H_{12}$ ($2.2966$)}
\put(73,100){\line(0,-1){6}}
\put(73,94){\circle*{2}}
\put(80,100){\line(1,0){7}}
\put(87,100){\circle*{2}}
\put(52,100){\line(0,-1){6}}
\put(52,94){\circle*{2}}
\put(66,100){\line(0,-1){6}}
\put(66,94){\circle*{2}}
\put(66,94){\line(0,-1){6}}
\put(66,88){\circle*{2}}
\put(66,100){\line(-1,-1){6}}
\put(60,94){\circle*{2}}
\put(1,100){\circle*{2}}
\put(1,100){\line(1,0){7}}
\put(8,100){\circle*{2}}
\put(8,100){\line(1,0){7}}
\put(15,100){\circle*{2}}
\put(15,100){\line(1,0){7}}
\put(22,100){\circle*{2}}
\put(22,100){\line(1,0){7}}
\put(29,100){\circle*{2}}
\put(2,77){$H_{11}$ ($2.3073$)}
\put(22,100){\line(0,-1){6}}
\put(22,94){\circle*{2}}
\put(29,100){\line(1,0){7}}
\put(36,100){\circle*{2}}
\put(22,94){\line(0,-1){6}}
\put(22,88){\circle*{2}}
\put(15,100){\line(0,-1){6}}
\put(15,94){\circle*{2}}
\put(15,94){\line(0,-1){6}}
\put(15,88){\circle*{2}}
\put(15,100){\line(-1,-1){6}}
\put(9,94){\circle*{2}}
\put(4,184){\circle*{2}}
\put(4,184){\line(1,0){7}}
\put(11,184){\circle*{2}}
\put(11,184){\line(1,0){7}}
\put(18,184){\circle*{2}}
\put(18,184){\line(1,0){7}}
\put(25,184){\circle*{2}}
\put(25,184){\line(1,0){7}}
\put(32,184){\circle*{2}}
\put(3,161){$H_{1}$ ($2.2361$)}
\put(18,184){\line(-1,-2){3}}
\put(18,184){\line(1,-2){3}}
\put(15,178){\circle*{2}}
\put(15,178){\line(0,-1){6}}
\put(15,172){\circle*{2}}
\put(21,178){\circle*{2}}
\put(21,178){\line(0,-1){6}}
\put(21,172){\circle*{2}}
\put(52,184){\circle*{2}}
\put(52,184){\line(1,0){7}}
\put(59,184){\circle*{2}}
\put(59,184){\line(1,0){7}}
\put(66,184){\circle*{2}}
\put(66,184){\line(1,0){7}}
\put(73,184){\circle*{2}}
\put(73,184){\line(1,0){7}}
\put(80,184){\circle*{2}}
\put(56,161){$H_{2}$ ($2.2059$)}
\put(80,184){\line(1,0){7}}
\put(87,184){\circle*{2}}
\put(66,184){\line(0,-1){6}}
\put(66,178){\circle*{2}}
\put(66,178){\line(0,-1){6}}
\put(66,172){\circle*{2}}
\put(66,184){\line(-1,-1){6}}
\put(60,178){\circle*{2}}
\put(103,184){\circle*{2}}
\put(103,184){\line(1,0){7}}
\put(110,184){\circle*{2}}
\put(110,184){\line(1,0){7}}
\put(117,184){\circle*{2}}
\put(117,184){\line(1,0){7}}
\put(124,184){\circle*{2}}
\put(124,184){\line(1,0){7}}
\put(131,184){\circle*{2}}
\put(107,161){$H_{3}$ ($2.1358$)}
\put(124,184){\line(0,-1){6}}
\put(124,178){\circle*{2}}
\put(131,184){\line(1,0){7}}
\put(138,184){\circle*{2}}
\put(124,178){\line(0,-1){6}}
\put(124,172){\circle*{2}}
\put(117,184){\line(0,-1){6}}
\put(117,178){\circle*{2}}
\put(1,142){\circle*{2}}
\put(1,142){\line(1,0){7}}
\put(8,142){\circle*{2}}
\put(8,142){\line(1,0){7}}
\put(15,142){\circle*{2}}
\put(15,142){\line(1,0){7}}
\put(22,142){\circle*{2}}
\put(22,142){\line(1,0){7}}
\put(29,142){\circle*{2}}
\put(2,119){$H_{6}$ ($2.0421$)}
\put(29,142){\line(1,0){7}}
\put(36,142){\circle*{2}}
\put(15,142){\line(0,-1){6}}
\put(15,136){\circle*{2}}
\put(15,136){\line(0,-1){6}}
\put(15,130){\circle*{2}}
\put(36,142){\line(0,-1){6}}
\put(36,136){\circle*{2}}
\put(160,184){\circle*{2}}
\put(160,184){\line(1,0){7}}
\put(167,184){\circle*{2}}
\put(167,184){\line(1,0){7}}
\put(174,184){\circle*{2}}
\put(174,184){\line(1,0){7}}
\put(181,184){\circle*{2}}
\put(181,184){\line(1,0){7}}
\put(188,184){\circle*{2}}
\put(155,161){$H_{4}$ ($2.1120$)}
\put(167,184){\line(0,-1){6}}
\put(167,178){\circle*{2}}
\put(153,184){\circle*{2}}
\put(153,184){\line(1,0){7}}
\put(153,184){\line(0,-1){6}}
\put(153,178){\circle*{2}}
\put(160,184){\line(0,-1){6}}
\put(160,178){\circle*{2}}
\put(212,184){\circle*{2}}
\put(212,184){\line(1,0){7}}
\put(219,184){\circle*{2}}
\put(219,184){\line(1,0){7}}
\put(226,184){\circle*{2}}
\put(226,184){\line(1,0){7}}
\put(233,184){\circle*{2}}
\put(233,184){\line(1,0){7}}
\put(240,184){\circle*{2}}
\put(207,161){$H_{5}$ ($2.0743$)}
\put(205,184){\circle*{2}}
\put(205,184){\line(1,0){7}}
\put(205,184){\line(0,-1){6}}
\put(205,178){\circle*{2}}
\put(212,184){\line(0,-1){6}}
\put(212,178){\circle*{2}}
\put(226,184){\line(0,-1){6}}
\put(226,178){\circle*{2}}
\put(59,142){\circle*{2}}
\put(59,142){\line(1,0){7}}
\put(66,142){\circle*{2}}
\put(66,142){\line(1,0){7}}
\put(73,142){\circle*{2}}
\put(73,142){\line(1,0){7}}
\put(80,142){\circle*{2}}
\put(80,142){\line(1,0){7}}
\put(87,142){\circle*{2}}
\put(58,119){$H_{7}$ ($2$)}
\put(52,142){\circle*{2}}
\put(52,142){\line(1,0){7}}
\put(52,142){\line(0,-1){6}}
\put(52,136){\circle*{2}}
\put(87,142){\line(0,-1){6}}
\put(87,136){\circle*{2}}
\put(59,142){\line(0,-1){6}}
\put(59,136){\circle*{2}}
\put(107,142){\circle*{2}}
\put(107,142){\line(1,0){7}}
\put(114,142){\circle*{2}}
\put(114,142){\line(1,0){7}}
\put(121,142){\circle*{2}}
\put(121,142){\line(1,0){7}}
\put(128,142){\circle*{2}}
\put(128,142){\line(1,0){7}}
\put(135,142){\circle*{2}}
\put(107,119){$H_8$ ($2.4495$)}
\put(121,142){\line(0,-1){6}}
\put(121,136){\circle*{2}}
\put(121,136){\line(0,-1){6}}
\put(121,130){\circle*{2}}
\put(115,136){\circle*{2}}
\put(115,136){\line(0,-1){6}}
\put(115,130){\circle*{2}}
\put(121,142){\line(1,-1){6}}
\put(127,136){\circle*{2}}
\put(127,136){\line(0,-1){6}}
\put(127,130){\circle*{2}}
\put(121,142){\line(-1,-1){6}}
\put(153,142){\circle*{2}}
\put(153,142){\line(1,0){7}}
\put(160,142){\circle*{2}}
\put(160,142){\line(1,0){7}}
\put(167,142){\circle*{2}}
\put(167,142){\line(1,0){7}}
\put(174,142){\circle*{2}}
\put(174,142){\line(1,0){7}}
\put(181,142){\circle*{2}}
\put(181,142){\line(1,0){7}}
\put(188,142){\circle*{2}}
\put(155,119){$H_{9}$ ($2.4236$)}
\put(167,142){\line(0,-1){6}}
\put(167,136){\circle*{2}}
\put(167,136){\line(0,-1){6}}
\put(167,130){\circle*{2}}
\put(167,142){\line(-1,-1){6}}
\put(161,136){\circle*{2}}
\put(161,136){\line(0,-1){6}}
\put(161,130){\circle*{2}}
\put(167,142){\line(1,-1){6}}
\put(173,136){\circle*{2}}
\put(205,142){\circle*{2}}
\put(205,142){\line(1,0){7}}
\put(212,142){\circle*{2}}
\put(212,142){\line(1,0){7}}
\put(219,142){\circle*{2}}
\put(219,142){\line(1,0){7}}
\put(226,142){\circle*{2}}
\put(226,142){\line(1,0){7}}
\put(233,142){\circle*{2}}
\put(207,119){$H_{10}$ ($2.3213$)}
\put(219,142){\line(-1,-2){3}}
\put(219,142){\line(1,-2){3}}
\put(216,136){\circle*{2}}
\put(216,136){\line(0,-1){6}}
\put(216,130){\circle*{2}}
\put(222,136){\circle*{2}}
\put(222,136){\line(0,-1){6}}
\put(222,130){\circle*{2}}
\put(226,142){\line(0,-1){6}}
\put(226,136){\circle*{2}}
\put(233,142){\line(1,0){7}}
\put(240,142){\circle*{2}}
\put(15,62){\circle*{2}}
\put(15,62){\line(1,0){7}}
\put(22,62){\circle*{2}}
\put(22,62){\line(1,0){7}}
\put(29,62){\circle*{2}}
\put(29,62){\line(1,0){7}}
\put(36,62){\circle*{2}}
\put(2,39){$H_{16}$ ($2.2143$)}
\put(8,62){\circle*{2}}
\put(8,62){\line(1,0){7}}
\put(29,62){\line(0,-1){6}}
\put(29,56){\circle*{2}}
\put(29,56){\line(0,-1){6}}
\put(29,50){\circle*{2}}
\put(1,62){\circle*{2}}
\put(1,62){\line(1,0){7}}
\put(1,62){\line(0,-1){6}}
\put(1,56){\circle*{2}}
\put(29,62){\line(-1,-1){6}}
\put(23,56){\circle*{2}}
\put(36,62){\line(0,-1){6}}
\put(36,56){\circle*{2}}
\put(52,62){\circle*{2}}
\put(52,62){\line(1,0){7}}
\put(59,62){\circle*{2}}
\put(59,62){\line(1,0){7}}
\put(66,62){\circle*{2}}
\put(66,62){\line(1,0){7}}
\put(73,62){\circle*{2}}
\put(73,62){\line(1,0){7}}
\put(80,62){\circle*{2}}
\put(53,39){$H_{17}$ ($2.2216$)}
\put(66,62){\line(0,-1){6}}
\put(66,56){\circle*{2}}
\put(66,56){\line(0,-1){6}}
\put(66,50){\circle*{2}}
\put(52,62){\line(0,-1){6}}
\put(52,56){\circle*{2}}
\put(59,62){\line(0,-1){6}}
\put(59,56){\circle*{2}}
\put(73,62){\line(0,-1){6}}
\put(73,56){\circle*{2}}
\put(80,62){\line(1,0){7}}
\put(87,62){\circle*{2}}
\put(110,62){\circle*{2}}
\put(110,62){\line(1,0){7}}
\put(117,62){\circle*{2}}
\put(117,62){\line(1,0){7}}
\put(124,62){\circle*{2}}
\put(124,62){\line(1,0){7}}
\put(131,62){\circle*{2}}
\put(131,62){\line(1,0){7}}
\put(138,62){\circle*{2}}
\put(104,40){$H_{18}$ ($2.2143$)}
\put(124,62){\line(0,-1){6}}
\put(124,56){\circle*{2}}
\put(124,56){\line(0,-1){6}}
\put(124,50){\circle*{2}}
\put(110,62){\line(0,-1){6}}
\put(110,56){\circle*{2}}
\put(117,62){\line(0,-1){6}}
\put(117,56){\circle*{2}}
\put(103,62){\circle*{2}}
\put(103,62){\line(1,0){7}}
\put(103,62){\line(0,-1){6}}
\put(103,56){\circle*{2}}
\put(160,62){\circle*{2}}
\put(160,62){\line(1,0){7}}
\put(167,62){\circle*{2}}
\put(167,62){\line(1,0){7}}
\put(174,62){\circle*{2}}
\put(174,62){\line(1,0){7}}
\put(181,62){\circle*{2}}
\put(181,62){\line(1,0){7}}
\put(188,62){\circle*{2}}
\put(155,40){$H_{19}$ ($2.1978$)}
\put(174,62){\line(0,-1){6}}
\put(174,56){\circle*{2}}
\put(167,62){\line(0,-1){6}}
\put(167,56){\circle*{2}}
\put(153,62){\circle*{2}}
\put(153,62){\line(1,0){7}}
\put(153,62){\line(0,-1){6}}
\put(153,56){\circle*{2}}
\put(181,62){\line(0,-1){6}}
\put(181,56){\circle*{2}}
\put(188,62){\line(0,-1){6}}
\put(188,56){\circle*{2}}
\put(205,62){\circle*{2}}
\put(205,62){\line(1,0){7}}
\put(212,62){\circle*{2}}
\put(212,62){\line(1,0){7}}
\put(219,62){\circle*{2}}
\put(219,62){\line(1,0){7}}
\put(226,62){\circle*{2}}
\put(226,62){\line(1,0){7}}
\put(233,62){\circle*{2}}
\put(207,40){$H_{20}$ ($2.1566$)}
\put(219,62){\line(0,-1){6}}
\put(219,56){\circle*{2}}
\put(219,56){\line(0,-1){6}}
\put(219,50){\circle*{2}}
\put(205,62){\line(0,-1){6}}
\put(205,56){\circle*{2}}
\put(212,62){\line(0,-1){6}}
\put(212,56){\circle*{2}}
\put(233,62){\line(1,0){7}}
\put(240,62){\circle*{2}}
\put(240,62){\line(0,-1){6}}
\put(240,56){\circle*{2}}
\put(1,22){\circle*{2}}
\put(1,22){\line(1,0){7}}
\put(8,22){\circle*{2}}
\put(8,22){\line(1,0){7}}
\put(15,22){\circle*{2}}
\put(15,22){\line(1,0){7}}
\put(22,22){\circle*{2}}
\put(2,0){$H_{21}$ ($2.1358$)}
\put(8,22){\line(0,-1){6}}
\put(8,16){\circle*{2}}
\put(8,16){\line(0,-1){6}}
\put(8,10){\circle*{2}}
\put(1,22){\line(0,-1){6}}
\put(1,16){\circle*{2}}
\put(22,22){\line(1,0){7}}
\put(29,22){\circle*{2}}
\put(22,22){\line(0,-1){6}}
\put(22,16){\circle*{2}}
\put(22,16){\line(0,-1){6}}
\put(22,10){\circle*{2}}
\put(29,22){\line(1,0){7}}
\put(36,22){\circle*{2}}
\put(110,22){\circle*{2}}
\put(110,22){\line(1,0){7}}
\put(117,22){\circle*{2}}
\put(117,22){\line(1,0){7}}
\put(124,22){\circle*{2}}
\put(124,22){\line(1,0){7}}
\put(131,22){\circle*{2}}
\put(131,22){\line(1,0){7}}
\put(138,22){\circle*{2}}
\put(104,0){$H_{23}$ ($2.1606$)}
\put(103,22){\circle*{2}}
\put(103,22){\line(1,0){7}}
\put(103,22){\line(0,-1){6}}
\put(103,16){\circle*{2}}
\put(131,22){\line(0,-1){6}}
\put(131,16){\circle*{2}}
\put(138,22){\line(0,-1){6}}
\put(138,16){\circle*{2}}
\put(117,22){\line(0,-1){6}}
\put(117,16){\circle*{2}}
\put(110,22){\line(0,-1){6}}
\put(110,16){\circle*{2}}
\put(160,22){\circle*{2}}
\put(160,22){\line(1,0){7}}
\put(167,22){\circle*{2}}
\put(167,22){\line(1,0){7}}
\put(174,22){\circle*{2}}
\put(174,22){\line(1,0){7}}
\put(181,22){\circle*{2}}
\put(181,22){\line(1,0){7}}
\put(188,22){\circle*{2}}
\put(155,0){$H_{24}$ ($2.0922$)}
\put(153,22){\circle*{2}}
\put(153,22){\line(1,0){7}}
\put(153,22){\line(0,-1){6}}
\put(153,16){\circle*{2}}
\put(181,22){\line(0,-1){6}}
\put(181,16){\circle*{2}}
\put(188,22){\line(0,-1){6}}
\put(188,16){\circle*{2}}
\put(160,22){\line(0,-1){6}}
\put(160,16){\circle*{2}}
\put(181,16){\line(0,-1){6}}
\put(181,10){\circle*{2}}
\put(62,22){\circle*{2}}
\put(62,22){\line(1,0){7}}
\put(69,22){\circle*{2}}
\put(69,22){\line(1,0){7}}
\put(76,22){\circle*{2}}
\put(76,22){\line(1,0){7}}
\put(83,22){\circle*{2}}
\put(83,22){\line(1,0){7}}
\put(90,22){\circle*{2}}
\put(52,0){$H_{22}$ ($2.1224$)}
\put(76,22){\line(0,-1){6}}
\put(76,16){\circle*{2}}
\put(55,22){\circle*{2}}
\put(55,22){\line(1,0){7}}
\put(83,22){\line(0,-1){6}}
\put(83,16){\circle*{2}}
\put(90,22){\line(0,-1){6}}
\put(90,16){\circle*{2}}
\put(48,22){\circle*{2}}
\put(48,22){\line(1,0){7}}
\put(48,22){\line(0,-1){6}}
\put(48,16){\circle*{2}}
\put(217,22){\circle*{2}}
\put(217,22){\line(1,0){7}}
\put(224,22){\circle*{2}}
\put(224,22){\line(1,0){7}}
\put(231,22){\circle*{2}}
\put(231,22){\line(1,0){7}}
\put(238,22){\circle*{2}}
\put(238,22){\line(1,0){7}}
\put(245,22){\circle*{2}}
\put(207,0){$H_{25}$ ($2.0529$)}
\put(210,22){\circle*{2}}
\put(210,22){\line(1,0){7}}
\put(238,22){\line(0,-1){6}}
\put(238,16){\circle*{2}}
\put(245,22){\line(0,-1){6}}
\put(245,16){\circle*{2}}
\put(203,22){\circle*{2}}
\put(203,22){\line(1,0){7}}
\put(203,22){\line(0,-1){6}}
\put(203,16){\circle*{2}}
\put(210,22){\line(0,-1){6}}
\put(210,16){\circle*{2}}
\end{picture}

  \caption{\footnotesize Trees $H_{1}, \ldots, H_{25}$ and  their spectral radii.}\label{H9-11}
\end{figure}

In what follows, assume $n\ge13$.
By Lemmas \ref{caterpillar}, \ref{max-degree} and \ref{diameter},
$T^*$ is a caterpillar with  $\Delta(T^*)=3$ and $d(T^*)=\frac{n+5}{2}$.
Let $P=v_1v_2\cdots v_{\frac{n+5}{2}}v_{\frac{n+7}{2}}$ be a diametrical path of $T^*$.
Note that $d_{T^*}(v_2)=d_{T^*}(v_{\frac{n+5}{2}})=2$ by Lemma \ref{leave-number}.
Then
$T^*$ is obtained from the  diametrical path $P$ by
adding $n-\frac{n+7}{2}=\frac{n-7}{2}$ leaves to
 $\frac{n-7}{2}$ vertices in $\{v_3,v_4,\ldots, v_{\frac{n+3}{2}}\}$.
  Since $\gamma(T^*)=\frac{n-1}{2}$, we have
 $T^*\cong T^{\frac{n+3}{2}}_{i,\frac{n-3}{2}-i}$ (see Fig.\ref{T}) for some $i\in[\lceil\frac{n-3}{4}\rceil, \frac{n-3}{2}]$.
In order to complete the proof, it suffices to show the following claim.

\begin{clm}\label{claim-2}$\rho(T^{\frac{n+3}{2}}_{i,\frac{n-3}{2}-i})<\rho(T^{\frac{n+3}{2}}_{i+1,\frac{n-3}{2}-i-1})$
for any $i\in[\lceil\frac{n-3}{4}\rceil, \frac{n-5}{2}]$.
\end{clm}
\begin{proof}
For convenience, denote $j=\frac{n-3}{2}-i$ and $d\!+\!1=\frac{n+3}{2}$, that is,
$T^{\frac{n+3}{2}}_{i,\frac{n-3}{2}-i}=T^{d+1}_{i,j}$ and $T^{\frac{n+3}{2}}_{i+1,\frac{n-3}{2}-i-1}=T^{d+1}_{i+1,j-1}$.
By Lemma \ref{qubian}, we obtain
\begin{equation}\label{poly1}
f(T^{d+1}_{i,j},x )=f(T^{d}_{i,j-1},x) f(P_2,x) -f(T^{d-1}_{i,j-2},x)f^2(P_1,x)
\end{equation}
and
\begin{equation}\label{poly2}
f(T^{d+1}_{i+1,j-1},x )=f(T^{d}_{i,j-1},x) f(P_2,x) -f(T^{d-1}_{i-1,j-1},x)f^2(P_1,x).
\end{equation}
Combining (\ref{poly1}) and (\ref{poly2}), we have
\begin{equation}\label{cha1}
f(T^{d+1}_{i,j},x )-f(T^{d+1}_{i+1,j-1},x )= [f(T^{d-1}_{i-1,j-1},x)-f(T^{d-1}_{i,j-2},x)]x^2.
\end{equation}
By Lemma \ref{qubian} again, we get
\begin{equation}\label{poly3}
f(T^{d-1}_{i-1,j-1},x )=f(T^{d-2}_{i-1,j-2},x) f(P_2,x) -f(T^{d-3}_{i-1,j-3},x)f^2(P_1,x)
\end{equation}
and
\begin{equation}\label{poly4}
f(T^{d-1}_{i,j-2},x )=f(T^{d-2}_{i-1,j-2},x) f(P_2,x) -f(T^{d-3}_{i-2,j-2},x)f^2(P_1,x).
\end{equation}
Combining (\ref{cha1}), (\ref{poly3}) and (\ref{poly4}), we have
$$ f(T^{d+1}_{i,j},x )-f(T^{d+1}_{i+1,j-1},x )= [f(T^{d-3}_{i-2,j-2},x)-f(T^{d-3}_{i-1,j-3},x)]x^4.$$
Using Lemma \ref{qubian} $(j-1)$ times as similar as above, we obtain
\begin{equation}\label{cha2}
 f(T^{d+1}_{i,j},x )-f(T^{d+1}_{i+1,j-1},x )= [f(T^{d+1-(2j-2)}_{i-j+1,1},x)-f(T^{d+1-(2j-2)}_{i-j+2,0},x)]x^{2j-2}.
 \end{equation}

\begin{figure}[h]
  \centering
  \footnotesize
 % This is a LaTeX picture output by TeXCAD.
% File name: [Clipboard].
% Version of TeXCAD: 4.3
% Reference / build: 30-Jun-2012 (rev. 105)
% For new versions, check: http://texcad.sf.net/
% Options on the following lines.
%\grade{\on}
%\emlines{\off}
%\epic{\off}
%\beziermacro{\on}
%\reduce{\on}
%\snapping{\on}
%\pvinsert{% Your \input, \def, etc. here}
%\quality{8.000}
%\graddiff{0.005}
%\snapasp{1}
%\zoom{9.5137}
\unitlength 0.9mm % = 2.845pt
\linethickness{0.4pt}
\ifx\plotpoint\undefined\newsavebox{\plotpoint}\fi % GNUPLOT compatibility
\begin{picture}(146,17)(0,0)
\put(82,15){\circle*{2}}
\put(82,15){\line(0,-1){8}}
\put(82,7){\circle*{2}}
\put(82,15){\line(1,0){8}}
\put(90,15){\circle*{2}}
\put(90,15){\line(0,-1){8}}
\put(90,7){\circle*{2}}
\put(106,15){\circle*{2}}
\put(106,15){\line(0,-1){8}}
\put(106,7){\circle*{2}}
\put(106,15){\line(1,0){8}}
\put(114,15){\circle*{2}}
\put(114,15){\line(1,0){8}}
\put(122,15){\circle*{2}}
%\dashline{1}(90,15)(106,15)
\put(89.93,14.93){\line(1,0){.9412}}
\put(91.812,14.93){\line(1,0){.9412}}
\put(93.694,14.93){\line(1,0){.9412}}
\put(95.577,14.93){\line(1,0){.9412}}
\put(97.459,14.93){\line(1,0){.9412}}
\put(99.341,14.93){\line(1,0){.9412}}
\put(101.224,14.93){\line(1,0){.9412}}
\put(103.106,14.93){\line(1,0){.9412}}
\put(104.989,14.93){\line(1,0){.9412}}
%\end
\put(104,17){$v_{i\!-\!j\!+\!1}$}
\put(81,17){$v_{1}$}
\put(89,17){$v_{2}$}
\put(122,15){\line(1,0){8}}
\put(130,15){\circle*{2}}
\put(97,0){$T^{d+1-(2j-2)}_{i-j+1,1}$}
\put(129,15){\line(1,0){8}}
\put(137,15){\circle*{2}}
\put(137,15){\line(1,0){8}}
\put(145,15){\circle*{2}}
\put(1,15){\circle*{2}}
\put(1,15){\line(0,-1){8}}
\put(1,7){\circle*{2}}
\put(1,15){\line(1,0){8}}
\put(9,15){\circle*{2}}
\put(9,15){\line(0,-1){8}}
\put(9,7){\circle*{2}}
\put(25,15){\circle*{2}}
\put(25,15){\line(0,-1){8}}
\put(25,7){\circle*{2}}
\put(25,15){\line(1,0){8}}
\put(33,15){\circle*{2}}
\put(33,15){\line(0,-1){8}}
\put(33,7){\circle*{2}}
\put(33,15){\line(1,0){8}}
\put(41,15){\circle*{2}}
%\dashline{1}(9,15)(25,15)
\put(8.93,14.93){\line(1,0){.9412}}
\put(10.812,14.93){\line(1,0){.9412}}
\put(12.694,14.93){\line(1,0){.9412}}
\put(14.577,14.93){\line(1,0){.9412}}
\put(16.459,14.93){\line(1,0){.9412}}
\put(18.341,14.93){\line(1,0){.9412}}
\put(20.224,14.93){\line(1,0){.9412}}
\put(22.106,14.93){\line(1,0){.9412}}
\put(23.989,14.93){\line(1,0){.9412}}
%\end
\put(32,17){$v_{i\!-\!j\!+\!2}$}
\put(23,17){$v_{i\!-\!j\!+\!1}$}
\put(0,17){$v_{1}$}
\put(8,17){$v_{2}$}
\put(15,0){$T^{d+1-(2j-2)}_{i-j+2,0}$}
\put(41,15){\line(1,0){8}}
\put(49,15){\circle*{2}}
\put(49,15){\line(1,0){8}}
\put(57,15){\circle*{2}}
\end{picture}

  \caption{\footnotesize Graphs $T^{d+1-(2j-2)}_{i-j+2,0}$ and $T^{d+1-(2j-2)}_{i-j+1,1}$. }\label{T-compare}
\end{figure}
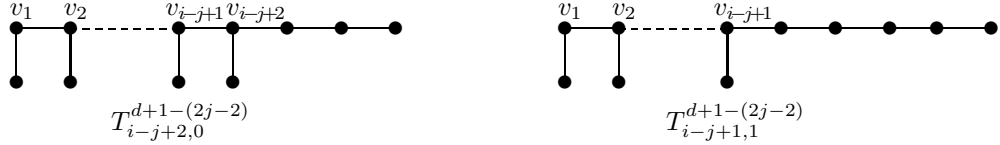
Clearly,
$T^{d+1-(2j-2)}_{i-j+1,1}$ is obtained from $T^{d+1-(2j-2)}_{i-j+2,0}$ by subdividing the edge $v_{i-j+1}v_{i-j+2}$ and deleting the leaf adjacent to $v_{i-j+2}$ (see  Fig.\ref{T-compare}).
By Lemmas \ref{subdivision} and \ref{subgraph},
we have $$\rho(T^{d+1-(2j-2)}_{i-j+1,1})< \rho(T^{d+1-(2j-2)}_{i-j+2,0}).$$
By Lemma \ref{compare}, we have
$f(T^{d+1-(2j-2)}_{i-j+1,1},x)>f(T^{d+1-(2j-2)}_{i-j+2,0},x)$ for  $x\ge \rho(T^{d+1-(2j-2)}_{i-j+1,1})$.
From (\ref{cha2}) we get
$$f(T^{d+1}_{i,j},x )>f(T^{d+1}_{i+1,j-1},x )$$
for any $x\ge \rho(T^{d+1}_{i-j+1,1})$.
Thus, $\rho(T^{d+1}_{i,j})<\rho(T^{d+1}_{i+1,j-1})$, and so
the claim holds.
\end{proof}

From Claim \ref{claim-2}, $T^{\frac{n+3}{2}}_{\lceil\frac{n-3}{4}\rceil, \lfloor\frac{n-3}{4}\rfloor }$ has the minimum spectral radius among all graphs in
$\{T^{\frac{n+3}{2}}_{i,\frac{n-3}{2}-i}: \ i\in[\lceil\frac{n-3}{4}\rceil, \frac{n-3}{2}]\}$.
Thus, $T^* \cong T^{\frac{n+3}{2}}_{\lceil\frac{n-3}{4}\rceil, \lfloor\frac{n-3}{4}\rfloor }$ for $n\ge3$.
\end{proof}

\section*{Acknowledgements}
The authors would like to express their sincere gratitude to the referees for their very careful reading of the paper and
for insightful comments and valuable suggestions, which improved the presentation of this paper.

\end{document}